\documentclass[10pt]{article}
\usepackage{amsfonts}
\usepackage{amsmath}
\usepackage{amsthm}
\usepackage{hyperref}

\hypersetup{
  pdftitle   = {},
  pdfauthor  = {},
  pdfcreator = {\LaTeX\ with package \flqq hyperref\frqq}
}
\DeclareMathAlphabet{\mathpzc}{OT1}{pzc}{m}{it}
\newtheorem{theorem}{Theorem}[section]
\newtheorem{proposition}[theorem]{Proposition}
\newtheorem{lemma}[theorem]{Lemma}
\newtheorem{corollary}[theorem]{Corollary}

\newtheorem{example}[theorem]{Example}

\newtheorem{remark}[theorem]{Remark}

\numberwithin{equation}{section}
\input{tcilatex}
\begin{document}

\title{Adjoint orbits of semi-simple Lie groups and Lagrangean submanifolds}
\author{Elizabeth Gasparim, Lino Grama and Luiz A. B. San Martin\thanks{%
EG was supported by Fapesp grant no. 2012/10179-5, LG by Fapesp grant no.
2012/07482-8, and LSM by CNPq grant no.\ 303755/09-1, Fapesp grant no.\
2012/18780-0 and CNPq/Universal grant no 476024/2012-9, Address: Imecc -
Unicamp, Departamento de Matem\'{a}tica. Rua S\'{e}rgio Buarque de Holanda,
651, Cidade Universit\'{a}ria Zeferino Vaz. 13083-859 Campinas - SP, Brasil.
E-mails: smartin@ime.unicamp.br, etgasparim@gmail.com, linograma@gmail.com.}}
\maketitle
\tableofcontents

\section{Introduction}

Let $G$ be a noncompact (real or complex) semi-simple Lie group with Lie
algebra $\mathfrak{g}$. The purpose of this paper is to study the
homogeneous spaces $G/Z_{H}$ with $Z_{H}$ the centralizer in $G$ of an
element $H$ belonging to a Cartan subalgebra of $\mathfrak{g}$.

Our motivation to study these homogeneous spaces is the construction of
Lefschetz fibrations in \cite{GGS}. The full description of these fibrations
requires a further understanding of the symplectic geometry (or rather
geometries) of $G/Z_{H}$, in particular those properties related to the
description of the Fukaya category of the Lagrangean vanishing cycles. In
this paper we get some of these properties that have independent interest.

To be more specific let $\mathfrak{a}$ be a Cartan--Chevalley algebra of $%
\mathfrak{g}$, that is, the Lie algebra of the $A$ component of an Iwasawa
decomposition $G=KAN$. We select a Weyl chamber $\mathfrak{a}^{+}\subset 
\mathfrak{a}$ and pick $H_{0}\in \mathrm{cl}\mathfrak{a}^{+}$. The adjoint
orbit $\mathrm{Ad}\left( G\right) H_{0}$ identifies with the homogeneous
space $G/Z_{H_{0}}$. Also the subadjoint orbit $\mathrm{Ad}\left( K\right)
H_{0}$ identifies with a flag manifold $\mathbb{F}_{H_{0}}=G/P_{H_{0}}$
where $P_{H_{0}}$ is the parabolic subgroup defined by $H_{0}$, which
contains $Z_{H_{0}}$.

In this paper we get other realizations of $G/Z_{H_{0}}$. First we prove
that $G/Z_{H_{0}}$ has the structure of a vector bundle over $\mathbb{F}%
_{H_{0}}$ isomorphic to the cotangent bundle $T^{\ast }\mathbb{F}_{H_{0}}$.
This fact was proved before by Azad-Van den Ban-Biswas \cite{ABB} using a
different approach. Here we exploit more decisively the associated vector
bundle construction obtained by $P_{H_{0}}$-representations by viewing $%
G\rightarrow \mathbb{F}_{H_{0}}=G/P_{H_{0}}$ as a $P_{H_{0}}$-principal
bundle (see Subsection \ref{subsecassocfibr}).

The isomorphism $\mathrm{Ad}\left( G\right) H_{0}\approx T^{\ast }\mathbb{F}%
_{H_{0}}$ provides the adjoint orbit with two different actions, namely the
natural transitive action on $\mathrm{Ad}\left( G\right) H_{0}$ and the
linear action on $T^{\ast }\mathbb{F}_{H_{0}}$ obtained by lifting the
action of $G$ on $\mathbb{F}_{H_{0}}$. The later action is not transitive
since the zero section is invariant. Thus one is asked to build a transitive
action on the cotangent bundle $T^{\ast }\mathbb{F}_{H_{0}}$ different from
the linear action. We do so by constructing a Lie algebra $\theta \left( 
\mathfrak{g}\right) $ of Hamiltonian vector fields (with respect to the
canonical symplectic form $\Omega $ of $T^{\ast }\mathbb{F}_{\Theta }$)
which is isomorphic to $\mathfrak{g}$. The elements of $\theta \left( 
\mathfrak{g}\right) $ are complete vector fields and hence the infinitesimal
action given by $\theta \left( \mathfrak{g}\right) $ integrates to an action
of a Lie group, by a classical theorem of Palais \cite{pal}. This action is
transitive and Hamiltonian by construction. The isotropy subgroup of the
transitive action is $Z_{H_{0}}$ and thus $T^{\ast }\mathbb{F}_{H_{0}}$ gets
identified with $G/Z_{H_{0}}$. It turns out that the moment map $\mu \colon
T^{\ast }\mathbb{F}_{H_{0}}\rightarrow \mathfrak{g}$ of the Hamiltonian
action takes values in $\mathrm{Ad}\left( G\right) H_{0}$ and is the inverse
of the previously defined map $\mathrm{Ad}\left( G\right) H_{0}\rightarrow
T^{\ast }\mathbb{F}_{H_{0}}$.

In another realization of $G/Z_{H_{0}}$, it is compactified to an algebraic
projective variety, namely the product $\mathbb{F}_{H_{0}}\times \mathbb{F}%
_{H_{0}^{\ast }}$ where $\mathbb{F}_{H_{0}^{\ast }}$ is the flag manifold
dual to $\mathbb{F}_{H_{0}}$ (see Section \ref{secsplit}). This is obtained
by the diagonal action $g\left( x,y\right) =\left( gx,gy\right) $ of $G$ on $%
\mathbb{F}_{H_{0}}\times \mathbb{F}_{H_{0}^{\ast }}$ which has just one open
and dense orbit whose isotropy group is $Z_{H_{0}}=Z_{H_{0}^{\ast }}$ and
hence realizes $G/Z_{H_{0}}$. The embedding $G/Z_{H_{0}}\rightarrow \mathbb{F%
}_{H_{0}}\times \mathbb{F}_{H_{0}^{\ast }}$ induces several geometric
structures on $G/Z_{H_{0}}$ inherited from those of $\mathbb{F}%
_{H_{0}}\times \mathbb{F}_{H_{0}^{\ast }}$. The point is that $\mathbb{F}%
_{H_{0}}\times \mathbb{F}_{H_{0}^{\ast }}$ is a flag manifold of $G\times G$
and hence admits Riemmannian metrics (Hermitian in the complex case)
invariant by the compact group $K\times K$. These metrics on $\mathbb{F}%
_{H_{0}}\times \mathbb{F}_{H_{0}^{\ast }}$ induce new metrics on $%
G/Z_{H_{0}} $, as well as new symplectic structures in the complex case.

The embedding $G/Z_{H_{0}}\rightarrow \mathbb{F}_{H_{0}}\times \mathbb{F}%
_{H_{0}^{\ast }}$ combined with representations of $\mathfrak{g}$ yields
realizations of $G/Z_{H_{0}}$ as orbits on $V\otimes V^{\ast }$ where $V$ is
the space of an irreducible representation of $\mathfrak{g}$ with highest
defined by $H_{0}$ (see Section \ref{secrepre}).

The last two realizations of $G/Z_{H_{0}}$ are used in Sections \ref{compact}
and \ref{seclagraph} to build a class of Lagrangean submanifolds in $%
G/Z_{H_{0}}$ with respect to the symplectic structures inherited from the
embedding $G/Z_{H_{0}}\rightarrow \mathbb{F}_{H_{0}}\times \mathbb{F}%
_{H_{0}^{\ast }}$.

\section{Adjoint orbits and cotangent bundles of flags}

Let $\mathfrak{g}$ be a noncompact semisimple Lie algebra (real or complex)
and let $G$ be a connected Lie group with finite centre and Lie algebra $%
\mathfrak{g}$ (for example $G$ may be $\mathrm{Aut}_{0}\left( \mathfrak{g}%
\right) $, the component of the identity of the group of automorphisms).

Usual notation:

\begin{enumerate}
\item The Cartan decomposition: $\mathfrak{g}=\mathfrak{k}\oplus \mathfrak{s}
$, with global decomposition $G=KS$.

\item Iwasawa decomposition: $\mathfrak{g}=\mathfrak{k}\oplus \mathfrak{a}%
\oplus \mathfrak{n}$, with global decomposition $G=KAN$.

\item $\Pi $ is a set of roots of $\mathfrak{a}$, with a choice of a set of
positive roots $\Pi ^{+}$ and simple roots $\Sigma \subset \Pi ^{+}$ such
that $\mathfrak{n}^{+}=\sum_{\alpha >0}\mathfrak{g}_{\alpha }$ and $%
\mathfrak{g}_{\alpha }$ is the root space of the root $\alpha $. The
corresponding Weyl chamber is $\mathfrak{a}^{+}$.

\item A subset $\Theta \subset \Sigma $ defines a parabolic subalgebra $%
\mathfrak{p}_{\Theta }$ with parabolic subgroup $P_{\Theta }$ and a flag $%
\mathbb{F}_{\Theta }=G/P_{\Theta }$. The flag is also $\mathbb{F}_{\Theta
}=K/K_{\Theta }$, where $K_{\Theta }=K\cap P_{\Theta }$. The Lie algebra of $%
K_{\Theta }$ is denoted $\mathfrak{k}_{\Theta }$.

\item $H_{\Theta }\in \mathrm{cl}\mathfrak{a}^{+}$ is \textit{characteristic}
for $\Theta \subset \Sigma $ if $\Theta =\{\alpha \in \Sigma :\alpha \left(
H_{\Theta }\right) =0\}$. Then, $\mathfrak{p}_{\Theta }=\bigoplus_{\lambda
\geq 0}\mathfrak{g}_{\lambda }$ where $\lambda $ runs through the
nonnegative eigenvalues of $\mathrm{ad}\left( H_{\Theta }\right) $.

Conversely, starting with $H_{0}\in \mathrm{cl}\mathfrak{a}^{+}$ we define $%
\Theta _{H_{0}}=\{\alpha \in \Sigma :\alpha \left( H_{0}\right) =0\}$ and in
the several objects requiring a subscript $\Theta $ we use $H_{0}$ instead
of $\Theta _{H_{0}}$. For instance, $\mathbb{F}_{H_{0}}=\mathbb{F}_{\Theta
_{H_{0}}}$, etc.

\item $b_{\Theta }=1\cdot K_{\Theta }=1\cdot P_{\Theta }$ denotes the origin
of the flag $\mathbb{F}_{\Theta }=K/K_{\Theta }=G/P_{\Theta }$.

\item We write 
\begin{equation*}
\mathfrak{n}_{\Theta }^{+}=\sum_{\alpha \left( H_{\Theta }\right) >0}%
\mathfrak{g}_{\alpha }\qquad \mathfrak{n}_{\Theta }^{-}=\sum_{\alpha \left(
H_{\Theta }\right) <0}\mathfrak{g}_{\alpha }
\end{equation*}%
so that $\mathfrak{g}=\mathfrak{n}_{\Theta }^{-}\oplus \mathfrak{z}_{\Theta
}\oplus \mathfrak{n}_{\Theta }^{+}$, where $\mathfrak{z}_{\Theta }$ is the
centralizer of $H_{\Theta }$ in $\mathfrak{g}$.

\item $Z_{\Theta }=\{g\in G:\mathrm{Ad}\left( g\right) H_{\Theta }=H_{\Theta
}\}$ is the centralizer in $G$ of the characteristic element $H_{\Theta }$.
Its Lie algebra is $\mathfrak{z}_{\Theta }$. Moreover, $K_{\Theta }$ is the
centralizer of $H_{\Theta }$ in $K$: 
\begin{equation*}
K_{\Theta }=Z_{K}\left( H_{\Theta }\right) =Z_{\Theta }\cap K=\{k\in K:%
\mathrm{Ad}\left( k\right) H_{\Theta }=H_{\Theta }\}.
\end{equation*}
\end{enumerate}

\begin{theorem}
\label{teodifeocotan}The adjoint orbit $\mathcal{O}\left( H_{\Theta }\right)
=\mathrm{Ad}\left( G\right) \cdot H_{\Theta }\approx G/Z_{\Theta }$ of the
characteristic element $H_{\Theta }$ is a $C^{\infty }$ vector bundle over $%
\mathbb{F}_{\Theta }$ isomorphic to the cotangent bundle $T^{\ast }\mathbb{F}%
_{\Theta }$. Moverover, we can write down a diffeomorphism $\iota :\mathrm{Ad%
}\left( G\right) \cdot H_{\Theta }\rightarrow T^{\ast }\mathbb{F}_{\Theta }$
such that

\begin{enumerate}
\item $\iota $ is equivariant with respect to the actions of $K$, that is,
for all $k\in K$, 
\begin{equation*}
\iota \circ \mathrm{Ad}\left( k\right) =\widetilde{k}\circ \iota
\end{equation*}%
where $\widetilde{k}$ is the lifting to $T^{\ast }\mathbb{F}_{\Theta }$ (via
the differential) of the action of $k$ on $\mathbb{F}_{\Theta }$.

\item The pullback of the canonical symplectic form on $T^{\ast }\mathbb{F}%
_{\Theta }$ by $\iota $ is the (real) Kirillov--Kostant--Souriaux form on
the orbit.
\end{enumerate}
\end{theorem}

The diffeomorphism $\iota \colon \mathcal{O}\left( H_{\Theta }\right)
\rightarrow T^{\ast }\mathbb{F}_{\Theta }$ will be defined in two steps,
namely $\mathcal{O}\left( H_{\Theta }\right) $ is proved to be diffeomorphic
to a vector bundle $\mathcal{V}\rightarrow K/K_{\Theta }$ associated to the
principal bundle $K\rightarrow K/K_{\Theta }$, built from a representation
of $K_{\Theta }$. Afterwards $\mathcal{V}\rightarrow K/K_{\Theta }$ is
proved to be isomorphic to $T^{\ast }\mathbb{F}_{\Theta }$. We write down
the expression for $\iota \colon\mathcal{O}\left( H_{\Theta }\right)
\rightarrow T^{\ast }\mathbb{F}_{\Theta }$ in (\ref{fordifeoexplicit}).

\begin{remark}
The equivariance of item (1) above holds only for the action of $K$.
However, there exists also an action of $G$ on the vector bundle, obtained
via the diffeomorphism with $\mathcal{O}\left( H_{\Theta }\right) $. Unlike
the action of $K$, this action is nonlinear since the linear action is not
transitive. We shall revisit the discussion of this action in terms of the
symplectic structure on the cotangent bundle $T^{\ast }\mathbb{F}_{\Theta }$.
\end{remark}

The projection $\pi \colon \mathcal{O}\left( H_{\Theta }\right) \rightarrow 
\mathbb{F}_{\Theta }$ is obtained via the action of $G$. Given the
homogeneous spaces $\mathcal{O}\left( H_{\Theta }\right) =G/Z_{\Theta }$ and 
$\mathbb{F}_{\Theta }=G/P_{\Theta }$, the centraliser $Z_{\Theta }$ is
contained in $P_{\Theta }$. Therefore, we obtain a canonical fibration $%
gZ_{\Theta }\mapsto gP_{\Theta }$ with fibre $P_{\Theta }/Z_{\Theta }$. On
one hand, in terms of the adjoint action the fibre is $\mathrm{Ad}\left(
P_{\Theta }\right) \cdot H_{\Theta }$, whereas on the other hand it is the
affine subspace $H_{\Theta }+\mathfrak{n}_{\Theta }^{+}$, where $\mathfrak{n}%
_{\Theta }^{+}$ is the sum of the eigenspaces of $\mathrm{ad}\left(
H_{\Theta }\right) $ associated to eigenvalues $>0$, that is, 
\begin{equation*}
\mathfrak{n}_{\Theta }^{+}=\sum \mathfrak{g}_{\alpha }
\end{equation*}%
with the sum running over the positive roots $\alpha $ outside $\langle
\Theta \rangle $, that is, with $\alpha \left( H_{\Theta }\right) >0$.
Indeed, if $g\in P_{\Theta }$ then $\mathrm{Ad}\left( g\right) H_{\Theta
}=H_{\Theta }+X$ with $X\in \mathfrak{n}_{\Theta }^{+}$. Moreover if $%
N_{\Theta }=\exp \mathfrak{n}_{\Theta }$ then the map $g\in N_{\Theta
}\mapsto \mathrm{Ad}\left( g\right) H_{\Theta }-H_{\Theta }\in \mathfrak{n}%
_{\Theta }$ is a diffeomorphism.

\begin{example}
The example of $\mathfrak{sl}\left( n\right) $ -- $\mathbb{R}$ or $\mathbb{C}
$ -- is enlightening: $P_{\Theta }$ is the group of matrices that are block
upper triangular. The diagonal part (in blocks) is $Z_{\Theta }$, whereas $%
\mathfrak{n}_{\Theta }^{+}$ is the upper triangular part above the blocks. $%
H_{\Theta }$ is a diagonal matrix that has one scalar matrix in each block.
Thus, conjugation $\mathrm{Ad}\left( g\right) H_{\Theta }=gH_{\Theta }g^{-1}$
keeps $H_{\Theta }$ inside the blocks and adds an upper triangular part
above the blocks, that is, $gH_{\Theta }g^{-1}=H_{\Theta }+X$ for some $X\in 
\mathfrak{n}_{\Theta }^{+}$.
\end{example}

The fibre of $\pi \colon\mathcal{O}\left( H_{\Theta }\right) \rightarrow 
\mathbb{F}_{\Theta }$ is a vector space. This alone does not guaranty the
structure of a vector bundle. Nevertheless, the structure of vector bundle
can be obtained as a bundle associated to the principal bundle $K\rightarrow
K/K_{\Theta }$ with structure group $K_{\Theta }$ (here we ought to endow
the flags with the structure of homogeneous spaces of $K$, not of $G$, since
the action of $G$ on the bundle is not linear).

\subsection{$\mathcal{O}\left( H_{\Theta }\right) \rightarrow \mathbb{F}%
_{\Theta }$ is a vector bundle \label{subsecassocfibr}}

The adjoint representation of $K_{\Theta }$ on $\mathfrak{g}$ leaves
invariant the subspace $\mathfrak{n}_{\Theta }^{+}$ since if $k\in K_{\Theta
}$ then $\mathrm{Ad}\left( k\right) $ commutes with $\mathrm{ad}\left(
H_{\Theta }\right) $, and consequently $\mathrm{Ad}\left( k\right) $ takes
eigenspaces of $\mathrm{ad}\left( H_{\Theta }\right) $ to eigenspaces.
Therefore $\mathrm{Ad}\left( k\right) $ leaves invariant $\mathfrak{n}%
_{\Theta }^{+}$, which is the sum of positive eigenspaces (with eigenvalues $%
>0$). It follows that the restriction of $\mathrm{Ad}$ defines a
representation $\rho $ of $K_{\Theta }$ on $\mathfrak{n}_{\Theta }^{+}$.
This allows us to define the vector bundle $K\times _{\rho }\mathfrak{n}%
_{\Theta }^{+}$ associated to the principal bundle $K\rightarrow K/K_{\Theta
}$. Now, to define a diffeomorphisms between $\mathcal{O}\left( H_{\Theta
}\right) $ and $K\times _{\rho }\mathfrak{n}_{\Theta }^{+}$ we recall that:

\begin{enumerate}
\item The elements of $K\times _{\rho }\mathfrak{n}_{\Theta }^{+}$ are
equivalence classes of pairs $\left( k,X\right) $ of the equivalence
relation $\left( ka,\rho \left( a^{-1}\right) X\right) \sim \left(
k,X\right) $, $k\in K_{\Theta }$. We write the equivalence class of $\left(
k,X\right) \in K\times \mathfrak{n}_{\Theta }^{+}$ as $k\cdot X$. The group $%
K$ acts on $K\times _{\rho }\mathfrak{n}_{\Theta }^{+}$ by left translations.

\item $\mathcal{O}\left( H_{\Theta }\right) =\bigcup_{k\in K}\mathrm{Ad}%
\left( k\right) \left( H_{\Theta }+\mathfrak{n}_{\Theta }^{+}\right) $.
(That is, $\mathcal{O}\left( H_{\Theta }\right) $ is a union of affine
subspaces, analogous to the classical ruled surfaces.) This is a consequence
of the Iwasawa decomposition $G=KAN$. In fact, $AN\subset P_{\Theta }$, so $%
G=KP_{\Theta }$ and it follows that 
\begin{eqnarray*}
\mathrm{Ad}\left( G\right) H_{\Theta } &=&\mathrm{Ad}\left( K\right) \left(
H_{\Theta }+\mathfrak{n}_{\Theta }^{+}\right) \\
&=&\bigcup_{k\in K}\mathrm{Ad}\left( k\right) \left( H_{\Theta }+\mathfrak{n}%
_{\Theta }^{+}\right) .
\end{eqnarray*}
\end{enumerate}

\begin{proposition}
The map $\gamma \colon \mathcal{O}\left( H_{\Theta }\right) \rightarrow
K\times _{\rho }\mathfrak{n}_{\Theta }^{+}$ defined by 
\begin{equation*}
Y=\mathrm{Ad}\left( k\right) \left( H_{\Theta }+X\right) \in \mathcal{O}%
\left( H_{\Theta }\right) \mapsto k\cdot X\in K\times _{\rho }\mathfrak{n}%
_{\Theta }^{+}
\end{equation*}%
is a diffeomorphism satisfying

\begin{enumerate}
\item $\gamma $ is equivariant with respect to the actions of $K$.

\item $\gamma $ maps fibers onto fibers.

\item $\gamma $ maps the orbit $\mathrm{Ad}\left( K\right) H_{\Theta }$ onto
the zero section of $K\times _{\rho }\mathfrak{n}_{\Theta }^{+}$.
\end{enumerate}
\end{proposition}

\begin{proof}
We check that $\gamma $ is well defined:  if $\mathrm{Ad}%
\left( k\right) \left( H_{\Theta }+X\right) =\mathrm{Ad}\left( k_{1}\right)
\left( H_{\Theta }+X_{1}\right) $ then $\mathrm{Ad}\left( u\right) \left(
H_{\Theta }+X\right) =H_{\Theta }+X_{1}$ where $u=k_{1}^{-1}k$. By
equivariance, it then follows that 
\begin{eqnarray*}
u\cdot b_{\Theta } &=&u\cdot \pi \left( H_{\Theta }+X\right) \\
&=&\pi \left( \mathrm{Ad}\left( u\right) \left( H_{\Theta }+X\right) \right)
\\
&=&\pi \left( H_{\Theta }+X_{1}\right) \\
&=&b_{\Theta }.
\end{eqnarray*}%
Consequently $u\in K_{\Theta }$, therefore $\mathrm{Ad}\left( u\right)
\left( H_{\Theta }+X\right) =H_{\Theta }+\mathrm{Ad}\left( u\right)
X=H_{\Theta }+X_{1}$ with $X_{1}=\mathrm{Ad}\left( u\right) X$. Hence, 
\begin{equation*}
k_{1}\cdot X_{1}=ku^{-1}\cdot \rho \left( u\right) X=k\cdot X
\end{equation*}%
showing that $\gamma $ is well defined. It is surjective because $k\cdot
X=\gamma \left( \mathrm{Ad}\left( k\right) \left( H_{\Theta }+X\right)
\right) $. It is injective since $k_{1}\cdot X_{1}=k\cdot X$ implies $%
k_{1}=ku$ and $X_{1}=\mathrm{Ad}\left( u^{-1}\right) X$, $u\in K_{\Theta }$.
Hence, 
\begin{eqnarray*}
\mathrm{Ad}\left( k_{1}\right) \left( H_{\Theta }+X_{1}\right) &=&\mathrm{Ad}%
\left( k\right) \left( \mathrm{Ad}\left( u\right) H_{\Theta }+\mathrm{Ad}%
\left( u\right) X_{1}\right) \\
&=&\mathrm{Ad}\left( k\right) \left( H_{\Theta }+X\right) .
\end{eqnarray*}

Now, the fibre of $\mathcal{O}\left( H_{\Theta }\right) $ over $k\cdot
b_{\Theta }$ is \ $\mathrm{Ad}\left( k\right) \left( H_{\Theta }+\mathfrak{n}%
_{\Theta }^{+}\right) $, which is taken by $\gamma $ to elements of the type 
$k\cdot X$, that are in the fibre over $k\cdot b_{\Theta }$ of $K\times
_{\rho }\mathfrak{n}_{\Theta }^{+}$. Also, $\gamma \left( \mathrm{Ad}\left(
k\right) \left( H_{\Theta }\right) \right) =k\cdot 0$, which is in \ the
zero section of $K\times _{\rho }\mathfrak{n}_{\Theta }^{+}$. Equivariance
holds because 
\begin{equation*}
\gamma \circ \mathrm{Ad}\left( u\right) \left( \mathrm{Ad}\left( k\right)
\left( H_{\Theta }+X\right) \right) =\gamma \left( \mathrm{Ad}\left(
uk\right) \left( H_{\Theta }+X\right) \right) =uk\cdot X
\end{equation*}%
and the last term is the left action of $u\in K$ on the vector bundle.
Finally the diffeomorphism property follows from the manifold constructions
of $\mathcal{O}\left( H_{\Theta }\right) $ (as a homogeneous space) and $%
K\times _{\rho }\mathfrak{n}_{\Theta }^{+}$ (as an associated bundle).
\end{proof}

From the diffeomorphism $\gamma $ we endow $\mathcal{O}\left( H_{\Theta
}\right) $ with the structure of a vector bundle coming from $K\times _{\rho
}\mathfrak{n}_{\Theta }^{+}$. Its fibers are the affine subspaces $\mathrm{Ad%
}\left( k\right) \left( H_{\Theta }+\mathfrak{n}_{\Theta }^{+}\right) $ that
have vector space structure via the bijection with $\mathrm{Ad}\left(
k\right) \left( \mathfrak{n}_{\Theta }^{+}\right) $.

\subsection{Isomorphism with $T^{\ast }\mathbb{F}_{\Theta }$}

Firstly, let $L$ be a Lie group and $M\subset L$ be a closed subgroup, and
denote by $\iota \colon M\rightarrow \mathrm{Gl}\left( T_{x_{0}}\left(
L/M\right) \right) $ the isotropy representation of $M$ on the tangent space
of $L/M$ at the origin $x_{0}$. Then, the tangent bundle $T\left( L/M\right) 
$ is isomorphic to the vector bundle $L\times _{\iota }T_{x_{0}}\left(
L/M\right) $, associated to the principal bundle $L\rightarrow L/M$ via the
representation $\iota $. Similarly, if $\iota ^{\ast }$ is the dual
representation, then $T^{\ast }\left( L/M\right) $ is isomorphic to the
vector bundle $L\times _{\iota ^{\ast }}\left( T_{x_{0}}\left( L/M\right)
\right) ^{\ast }$. Secondly, observe that if $Q\times _{\rho _{1}}V$ and $%
Q\times _{\rho _{2}}W$ are two vector bundles associated to the principal
bundle $Q\rightarrow X$, via the representations $\rho _{1}$ and $\rho _{2}$
then $Q\times _{\rho _{1}}V$ is isomorphic to $Q\times _{\rho _{2}}W$ when $%
\rho _{1}$ and $\rho _{2}$ are equivalent representations.

Reaiming our focus to the flag $\mathbb{F}_{\Theta }$, note that the tangent
space to the origin $T_{b_{\Theta }}\mathbb{F}_{\Theta }$ is identified with 
$\mathfrak{n}_{\Theta }^{-}$, which is the subspace formed by the sum of
eigenspaces of $\mathrm{ad}\left( H_{\Theta }\right) $ associated to
negative eigenvalues, that is, 
\begin{equation*}
\mathfrak{n}_{\Theta }^{-}=\sum_{\alpha \left( H_{\Theta }\right) <0}%
\mathfrak{g}_{\alpha }.
\end{equation*}%
Via this identification, the isotropy representation becomes the restriction
of the adjoint representation.

The subspace $\mathfrak{n}_{\Theta }^{+}$ is isomorphic to the dual $\left( 
\mathfrak{n}_{\Theta }^{-}\right) ^{\ast }$ of $\mathfrak{n}_{\Theta }^{-}$
via the Cartan-Killing form $\langle \cdot ,\cdot \rangle $ of $\mathfrak{g}$%
. This means that the map 
\begin{equation*}
X\in \mathfrak{n}_{\Theta }^{+}\mapsto \langle X,\cdot \rangle \in \left( 
\mathfrak{n}_{\Theta }^{-}\right) ^{\ast }
\end{equation*}%
is an isomorphism. Once more, via this isomorphism the dual of the isotropy
representation becomes the representation $\rho $ given by restriction of
the adjoint representation.

Therefore, $T^{\ast }\mathbb{F}_{\Theta }=T^{\ast }\left( K/K_{\Theta
}\right) $ is isomorphic to $K\times _{\rho }\mathfrak{n}_{\Theta }^{+}$,
which in turn is diffeomorphic to the adjoint orbit $\mathcal{O}\left(
H_{\Theta }\right) $. Both diffeomorphisms permute the action of $K$. This
finishes the proof of the first part of Theorem \ref{teodifeocotan}, as well
as of item (1). Thus, the diffeomorphism $\iota \colon \mathcal{O}\left(
H_{\Theta }\right) \rightarrow T^{\ast }\mathbb{F}_{\Theta }$ is obtained by
composing $\gamma \colon \mathcal{O}\left( H_{\Theta }\right) \rightarrow
K\times _{\rho }\mathfrak{n}_{\Theta }^{+}$ with the vector bundle
isomorphism between $K\times _{\rho }\mathfrak{n}_{\Theta }^{+}$ and $%
T^{\ast }\mathbb{F}_{\Theta }$. It is explicitly given by 
\begin{equation}
\iota \colon \mathrm{Ad}\left( k\right) \left( H_{\Theta }+X\right) \in 
\mathcal{O}\left( H_{\Theta }\right) \mapsto \langle \mathrm{Ad}\left(
k\right) X,\cdot \rangle \in T_{kb_{\Theta }}^{\ast }\mathbb{F}_{\Theta }
\label{fordifeoexplicit}
\end{equation}%
where $X\in \mathfrak{n}_{\Theta }^{+}$ and $T_{kb_{\Theta }}\mathbb{F}%
_{\Theta }$ is identified with $\mathrm{Ad}\left( k\right) \mathfrak{n}%
_{\Theta }^{-}$.

Item (2) of Theorem \ref{teodifeocotan} will be a consequence of Proposition %
\ref{propmomento} below.

\subsection{The action of $G$ on $T^{\ast } \mathbb{F}_{\Theta }$}

The diffeomorphism $\iota \colon \mathcal{O}\left( H_{\Theta }\right)
\rightarrow T^{\ast }\mathbb{F}_{\Theta }$ induces an action of $G$ on $%
T^{\ast }\mathbb{F}_{\Theta }$ by $g\alpha =\iota \circ \mathrm{Ad}\left(
g\right) \circ \iota ^{-1}\left( \alpha \right) $, $g\in G$, $\alpha \in
T^{\ast }\mathbb{F}_{\Theta }$. The action of $K$ is linear since it is
given by the lifting of the linear action on $\mathbb{F}_{\Theta }$.
However, the action of $G$ is not linear because the linear action on $%
T^{\ast }\mathbb{F}_{\Theta }$ is not transitive (the zero section is
invariant). It is therefore natural to ask how does the action of $G$ behave
in terms of the geometry of $T^{\ast }\mathbb{F}_{\Theta }$. The description
of this action will be made via an infinitesimal action o the Lie algebra $%
\mathfrak{g}$ of $G$, that is, through a homomorphism $\theta \colon 
\mathfrak{g}\rightarrow \Gamma \left( T^{\ast }\mathbb{F}_{\Theta }\right) $%
, which associates to each element of the Lie algebra $\mathfrak{g}$ a
Hamiltonian vector field on $T^{\ast }\mathbb{F}_{\Theta }$.

Let $\Omega $ be the canonical symplectic form on $T^{\ast }\mathbb{F}%
_{\Theta }$. Given a vector field $X$ on $\mathbb{F}_{\Theta }$ denote by $%
X^{\#}$ the lifting of $X$ to $T^{\ast }\mathbb{F}_{\Theta }$. The flow of $%
X^{\#}$ is linear and is defined by $\alpha \in T_{x}^{\ast }\mathbb{F}%
_{\Theta }\mapsto \alpha \circ \left( d\phi _{-t}\right) _{\phi _{t}\left(
x\right) }$ where $\phi _{t}$ is the flow of $X$. The lifting satisfies:

\begin{enumerate}
\item $\pi _{\ast }X^{\#}=X$, where $\pi \colon T^{\ast }\mathbb{F}_{\Theta
}\mapsto \mathbb{F}_{\Theta }$ is the projection.

\item $X^{\#}$ is the Hamiltonian vector field with respect to $\Omega $ for
the function $h_{X}\left( \xi \right) =\xi \left( X\left( x\right) \right) $%
, $\xi \in T_{x}^{\ast }\mathbb{F}_{\Theta }$.

\item If $X$ and $Y$ are vector fields, then $[X,Y]^{\#}=[X^{\#},Y^{\#}]$,
that is, $X\mapsto X^{\#}$ is a homomorphism of Lie algebras.
\end{enumerate}

Now, for $Y\in \mathfrak{g}$ we denote the vector field on $\mathbb{F}%
_{\Theta }$ whose flow is $\exp tY$ by $\widetilde{Y}$ or simply by $Y$ if
there is no confusion.

Since the action of $K$ in $T^{\ast }\mathbb{F}_{\Theta }$ is linear, it
follows that the vector field induced by $A\in \mathfrak{k}$ on $T^{\ast }%
\mathbb{F}_{\Theta }$ is $X^{\#}$, that is, $\theta \left( X\right) =X^{\#}$
if $X\in \mathfrak{k}$. Using the Cartan decomposition $\mathfrak{g}=%
\mathfrak{k}\oplus \mathfrak{s}$, it remains to describe $\theta \left(
X\right) $ when $X\in \mathfrak{s}$. This is done modifying the vector field 
$X^{\#}$ by a vertical one so that the new vector field still projects on $X$%
.

The following lemma is well known. We include it here for the sake of
completeness.

\begin{lemma}
\label{lemvert}Let $M$ be a manifold and $f\colon M\rightarrow \mathbb{R}$.
Define $F\colon T^{\ast }M\rightarrow \mathbb{R}$ by $F=f\circ \pi $ ($\pi
\colon T^{\ast }M\rightarrow M$ is the projection). Let $V_{F}$ be the
Hamiltonian vector field of $F$ with respect to $\Omega $. Then, $V_{F}$ is
vertical ($\pi _{\ast }V_{F}=0$). Then $V_{F}$ is the constant parallel
vector field whose restriction to a fiber $T_{x}^{\ast }M$ is $-df_{x}\in
T_{x}^{\ast }M$.
\end{lemma}

\begin{proof}
A straightforward way to see this is to use local coordinates $q,p$ of $M$
and the fibre, respectively. Then, the Hamiltonian vector field is 
\begin{equation*}
V_{F}=\sum_{i}\frac{\partial F}{\partial p_{i}}\frac{\partial }{\partial
q_{i}}-\frac{\partial F}{\partial q_{i}}\frac{\partial }{\partial p_{i}}.
\end{equation*}%
Since the function $F$ does not depend on $p$, only the second term remains,
showing that the vector field is vertical. If $x=\left( q_{1,}\ldots
,q_{n}\right) \in M$ is fixed then the second term becomes 
\begin{equation*}
\sum_{i}-\frac{\partial F}{\partial q_{i}}\frac{\partial }{\partial p_{i}}%
=-df_{x}
\end{equation*}%
since $\partial F/\partial q_{i}=\partial f/\partial q_{i}$.
\end{proof}

We return to $\mathbb{F}_{\Theta }$ which coincides with the adjoint orbit $%
\mathrm{Ad}\left( K\right) \cdot H_{\Theta }\subset \mathfrak{s}$. Given $%
X\in \mathfrak{s}$, we can define the height function 
\begin{equation*}
f_{X}\left( x\right) =\langle x,X\rangle
\end{equation*}%
where $\langle \cdot ,\cdot \rangle $ is the Cartan-Killing form, which is
an inner product when restricted to $\mathfrak{s}$.

Now a choose a $K$-invariant Riemannian metric $\left( \cdot ,\cdot \right)
_{B}$ on $\mathbb{F}_{\Theta }$. The \ most convenient for our purposes is
the so called Borel metric which has the property that for any $X\in 
\mathfrak{s}$ the gradient of $f_{X}$ is exactly the vector field $X$
induced by $X$ (see Duistermat-Kolk-Varadarajan \cite{dkv}).

For $X\in \mathfrak{s}$ set $F_{X}=f_{X}\circ \pi $ and denote by $V_{X}$
its Hamiltonian vector field on $T^{\ast }\mathbb{F}_{\Theta }$. By the
lemma \ref{lemvert} $V_{X}$ is vertical.

The following lemma will be used to evaluate the symplectic form on the
several Hamiltonian vector fields defined above.

\begin{lemma}
We have the following directional derivatives:

\begin{enumerate}
\item If $A\in \mathfrak{k}$ and $X\in \mathfrak{s}$ then $%
A^{\#}F_{X}=F_{[A,X]}$.

\item If $X,Y\in \mathfrak{s}$ then $X^{\#}F_{Y}=Y^{\#}F_{X}$.

\item If $X,Y\in \mathfrak{s}$ then $V_{X}F_{Y}=0$.
\end{enumerate}
\end{lemma}

\begin{proof}
Since $\pi _{\ast }A^{\#}=A$ it follows that $A^{\#}F_{X}\left( \alpha
\right) =Af_{X}\left( x\right) $, $x=\pi \left( \alpha \right) $. But, 
\begin{eqnarray}
Af_{X}\left( x\right) &=&\frac{d}{dt}_{\left\vert t=0\right. }\langle 
\mathrm{Ad}\left( e^{tA}\right) x,X\rangle =\langle -[A,x],X\rangle
\label{forderivadaadjunta} \\
&=&\langle x,[A,X]\rangle  \notag \\
&=&f_{[A,X]}\left( x\right)  \notag
\end{eqnarray}%
showing the first equality.

Using again $\pi _{\ast }X^{\#}=X$ we get $X^{\#}F_{Y}\left( \alpha \right)
=Xf_{Y}\left( x\right) $, $x=\pi \left( \alpha \right) $. But, $Xf_{Y}\left(
x\right) =\left( X\left( x\right) ,Y\left( x\right) \right) _{B}$, since $Y$
is the gradient of $f_{Y}$ with respect to $\left( \cdot ,\cdot \right) _{B}$%
. By symmetry of the metric, $\left( X\left( x\right) ,Y\left( x\right)
\right) _{B}=Yf_{X}\left( x\right) $, proving the equality of (2). Finally, $%
F_{Y}$ is constant on the fibers and $V_{X}$ is vertical hence (3) follows.
\end{proof}

\begin{remark}
In the computation of the partial derivative of item (1) above we used the
fact that the Lie algebra of $G$ is formed by \textit{right} invariant
fields. For the bracket $\left[ \cdot ,\cdot \right] $ in $\mathfrak{g}$
formed by the right invariant vector fields the following equality holds $%
\mathrm{Ad}\left( e^{A}\right) =e^{-\mathrm{ad}\left( A\right) }$. Hence the
first equality of (\ref{forderivadaadjunta}). The reason to use right
invariant vector fields is so that we can project onto homogeneous spaces.
\end{remark}

Now we can compute the Lie brackets between the Hamiltonian vector fields.

\begin{corollary}
\label{corbrackets}We have the following Lie brackets:

\begin{enumerate}
\item If $A\in \mathfrak{k}$ and $X\in \mathfrak{s}$ then $%
[A^{\#},V_{X}]=V_{[A,X]}$.

\item If $X,Y\in \mathfrak{s}$ then $[X^{\#},V_{Y}]=[Y^{\#},V_{X}]$.

\item If $X,Y\in \mathfrak{s}$ then $[V_{X},V_{Y}]=0$.
\end{enumerate}
\end{corollary}

\begin{proof}
In fact, all vector fields involved are Hamiltonian. In general, on a
symplectic manifold if $Z$ and $W$ are the Hamiltonian vector fields of the
energy functions $u$ and $v$ then $[Z,W]$ is the Hamiltonian vector field of
the function $Zv$ (see \cite{abmars}, 3.3 -- Proposition 3.3.12 together
with Corollary 3.3.18). Combining this with the fact that $X\mapsto X^{\#}$
is a Lie algebra homomorphism, that is, $[X^{\#},Y^{\#}]=[X,Y]^{\#}$, we
obtain:

\begin{enumerate}
\item $[A^{\#},V_{X}]$ is the Hamiltonian vector field of the function $%
A^{\#}F_{X}=F_{\left[ A,X\right] }$, that is, $[A^{\#},V_{X}]=V_{[A,X]}$.

\item $[X^{\#},V_{Y}]$ is the Hamiltonian vector field of the function $%
X^{\#}F_{Y}=Y^{\#}F_{X}$. From which item (2) follows.

\item $[V_{X},V_{Y}]$ is the Hamiltonian vector field of the function $%
V_{X}F_{Y}=0$.
\end{enumerate}\end{proof}

\begin{corollary}
The map $\theta $ defined on $\mathfrak{g}$ and taking values on vector
fields of $T^{\ast }\mathbb{F}_{\Theta }$ defined by $\theta \left( A\right)
=A^{\#}$ if $A\in \mathfrak{k}$ and $\theta \left( X\right) =X^{\#}+V_{X}$
is a homomorphism of Lie algebras.
\end{corollary}

\begin{proof}
This follows directly from the brackets computed in Corollary \ref%
{corbrackets}.
\end{proof}

In other words, $\theta $ is an infinitesimal action of $\mathfrak{g}$ on $%
T^{\ast }\mathbb{F}_{\Theta }$. By a classical result of Palais this action
is integrated to an action of a connected Lie group $G$, whose Lie algebra
is $\mathfrak{g}$, provided the vector fields are complete.

\begin{lemma}
The vector fields $\theta \left( Z\right) $, $Z\in \mathfrak{g}$ are
complete.
\end{lemma}

\begin{proof}
Take $Z=A+X$ with $A\in \mathfrak{k}$ and $X\in \mathfrak{s}$ so that $%
\theta \left( Z\right) =A^{\#}+X^{\#}+V_{X}=\left( A+X\right) ^{\#}+V_{X}$.
Suppose by contradiction that there exists a maximal trajectory $z\left(
t\right) $ of $Z$ defined in a proper interval $\left( a,b\right) \subset 
\mathbb{R}$, with e.g. $b<\infty $. This implies that $\lim_{t\rightarrow
b}z\left( t\right) =\infty $. Let $x\left( t\right) $ be the projection of $%
z\left( t\right) $ onto $\mathbb{F}_{\Theta }$. Then $x\left( t\right) $ is
a trajectory of the vector field $\widetilde{A+X}$ on $\mathbb{F}_{\Theta }$
induced by $A+X$. Since $\widetilde{A+X}$ is complete (by compactness of $%
\mathbb{F}_{\Theta }$) there exists $\lim_{t\rightarrow b}x\left( t\right)
=x\left( b\right) $.

In a local trivialization $T^{\ast }\mathbb{F}_{\Theta }\approx U\times 
\mathbb{R}^{n}$ around $x\left( b\right) $ we have $z\left( t\right) =\left(
x\left( t\right) ,y\left( t\right) \right) $. The second component $y\left(
t\right) $ satisfies a linear equation 
\begin{equation*}
\dot{y}=A\left( t\right) y+c\left( t\right)
\end{equation*}%
where $A\left( t\right) $ is the derivative at $x\left( t\right) $ of the
vector field $\widetilde{A+X}$ and $c\left( t\right) =V_{X}\left( x\left(
t\right) \right) $. The solution of this linear equation is defined in a
neighborhood of $b$, contradicting the fact that $z\left( t\right)
\rightarrow \infty $ as $t\rightarrow b$.
\end{proof}

As a consequence we obtain the following result.

\begin{proposition}
The infinitesimal action $\theta $ integrates to an action $a\colon G\times
T^{\ast }\mathbb{F}_{\Theta }\rightarrow T^{\ast }\mathbb{F}_{\Theta }$ of a
connected Lie group $G$ with Lie algebra $\mathfrak{g}$. This action $%
a\left( g,x\right) =g\cdot x$ satisfies:

\begin{enumerate}
\item $\theta \left( Y\right) \left( x\right) =\frac{d}{dt}_{\left\vert
t=0\right. }a\left( e^{tY},x\right) $ for all $Y\in \mathfrak{g}$.

\item The action is Hamiltonian since the vector fields $\theta \left(
Y\right) $, $Y\in \mathfrak{g}$ are Hamiltonian vector fields.

\item The projection $\pi \colon T^{\ast }\mathbb{F}_{\Theta }\rightarrow 
\mathbb{F}_{\Theta }$ is equivariant with respect to this new action and the
action of $G$ on $\mathbb{F}_{\Theta }$.

\item The action $a$ is transitive.
\end{enumerate}
\end{proposition}

\begin{proof}
The first two items are due to the construction of $\theta $ and $a$. As to
equivariance it holds because for any $Y\in \mathfrak{g}$ the projection $%
\pi _{\ast }\theta \left( Y\right) $ is the vector field $\widetilde{Y}$
induced by $Y$ via the action on $\mathbb{F}_{\Theta }$.

To prove transitivity we observe that the Cartan decomposition $\mathfrak{g}=%
\mathfrak{k}\oplus \mathfrak{s}$ induces the Cartan decomposition $G=KS$.
The group $K$ acts on $T^{\ast }\mathbb{F}_{\Theta }$ by linear
transformations among the fibres, since $\theta \left( A\right) =A^{\#}$ for 
$A\in \mathfrak{k}$. Since $K$ acts transitively on $\mathbb{F}_{\Theta }$,
it suffices to verify that $G$ acts transitively on a single fibre.

Let $b_{\Theta }\in \mathbb{F}_{\Theta }$ be the origin of $\mathbb{F}%
_{\Theta }$ also seen as the null vector of $T_{b_{\Theta }}^{\ast }\mathbb{F%
}_{\Theta }$. Then the orbit $G\cdot b_{\Theta }$ on $T^{\ast }\mathbb{F}%
_{\Theta }$ is open, because the tangent space to the orbit 
\begin{equation*}
\{\theta \left( Z\right) \left( b_{\Theta }\right) :Z\in \mathfrak{g}\}
\end{equation*}%
coincides with the tangent space $T_{b_{\Theta }}\left( T^{\ast }\mathbb{F}%
_{\Theta }\right) $.

In fact, \ $T_{b_{\Theta }}\left( T^{\ast }\mathbb{F}_{\Theta }\right) $ is
the sum of the (horizontal) tangent space $T\mathbb{F}_{\Theta }$ with the
(vertical) fibre $T_{b_{\Theta }}^{\ast }\mathbb{F}_{\Theta }$. The
transitive action of $K$ on $\mathbb{F}_{\Theta }$ guaranties that $T\mathbb{%
F}_{\Theta }=\{\theta \left( A\right) \left( b_{\Theta }\right) :A\in 
\mathfrak{k}\}$. On the other hand, given $X\in \mathfrak{s}$ there exists $%
A\in \mathfrak{k}$ such that $\widetilde{X}\left( b_{\Theta }\right) =%
\widetilde{A}\left( b_{\Theta }\right) $. In such case, $\widetilde{X-A}%
\left( b_{\Theta }\right) =0$, which implies that $\left( X-A\right)
^{\#}\left( b_{\Theta }\right) =V_{X}\left( b_{\Theta }\right) $. The
vertical vector $V_{X}\left( b_{\Theta }\right) $ is the linear functional of 
$T_{b_{\Theta }}\mathbb{F}_{\Theta }$ given by $v\mapsto \left( \widetilde{X}%
\left( b_{\Theta }\right) ,v\right) _{B}=\left( df_{X}\right) _{b_{\Theta
}}\left( v\right) $. These linear functionals generate $T_{b_{\Theta
}}^{\ast }\mathbb{F}_{\Theta }$ since $\widetilde{X}\left( b_{\Theta
}\right) $, $X\in \mathfrak{s}$, generates $T_{b_{\Theta }}\mathbb{F}%
_{\Theta }$. This shows that the vertical space is contained in the space
tangent to the orbit, concluding the proof that the orbit is open.

Finally, take $H\in \mathfrak{a}^{+}$. Then, $V_{H}\left( b_{\Theta }\right)
=0$ since $\widetilde{H}\left( b_{\Theta }\right) =0$. Moreover, $H^{\#}$ is
vertical in the fibre over $b_{\Theta }$ and restricts to the fibre as a
linear vector field. Since $H$ was chosen in the positive chamber $\mathfrak{%
a}^{+}$, such linear vector field is given by a linear transformation whose
eigenvalues are all negative. This implies that any trajectory of $H^{\#}$
in the fibre intercepts every neighborhood of the origin. Since $G\cdot
b_{\Theta }$ contains a neighborhood of the origin we conclude that $G$ is
transitive in the fibre $T_{b_{\Theta }}^{\ast }\mathbb{F}_{\Theta }$,
showing that the action is transitive.
\end{proof}

The next step is to identify $T^{\ast }\mathbb{F}_{\Theta }$ as a
homogeneous space of $G$, via the transitive action of the previous
proposition. First of all we shall find the isotropy algebra $\mathfrak{l}$
at $b_{\Theta }$, that is, 
\begin{equation*}
\mathfrak{l}=\{Y\in \mathfrak{g}:\theta \left( Y\right) \left( b_{\Theta
}\right) =0\}
\end{equation*}%
where the origin of the flag $b_{\Theta }$ is seen also the null vector of $%
T_{b_{\Theta }}^{\ast }\mathbb{F}_{\Theta }$.

\begin{lemma}
The isotropy subalgebra $\mathfrak{l}=\{Y\in \mathfrak{g}:\theta \left(
Y\right) \left( b_{\Theta }\right) =0\}$ coincides with the isotropy
subalgebra at $H_{\Theta }$ of the adjoint orbit, that is, $\mathfrak{l}=%
\mathfrak{z}_{\Theta }$.
\end{lemma}

\begin{proof}
Let $Y\in \mathfrak{g}$ with $\theta \left( Y\right) \left( b_{\Theta
}\right) =0$ and $Y=A+X$, $A\in \mathfrak{k}$ and $X\in \mathfrak{s}$. Then, 
$\theta \left( Y\right) =A^{\#}+X^{\#}+V_{X}$ and since $A^{\#}\left(
b_{\Theta }\right) =X^{\#}\left( b_{\Theta }\right) =0$, it follows that $%
V_{X}\left( b_{\Theta }\right) =0$. However, as in the previous proof, $%
V_{X}\left( b_{\Theta }\right) $ is the linear functional $v\mapsto \left( 
\widetilde{X}\left( b_{\Theta }\right) ,v\right) _{B}$. Therefore, $%
\widetilde{X}\left( b_{\Theta }\right) =0$. On the other hand, $\theta
\left( Y\right) \left( b_{\Theta }\right) =0$ implies that $\widetilde{Y}%
\left( b_{\Theta }\right) =0$, and consequently $\widetilde{A}\left(
b_{\Theta }\right) =-\widetilde{X}\left( b_{\Theta }\right) =0$. This shows
that $A\in \mathfrak{p}_{\Theta }\cap \mathfrak{k}\subset \mathfrak{z}%
_{\Theta }$ and $B\in \mathfrak{p}_{\Theta }\cap \mathfrak{s}\subset 
\mathfrak{z}_{\Theta }$, thus $Y\in \mathfrak{z}_{\Theta }$. Therefore, 
\begin{equation*}
\{Y\in \mathfrak{g}:\theta \left( Y\right) \left( b_{\Theta }\right)
=0\}\subset \mathfrak{z}_{\Theta }.
\end{equation*}%
Equality follows from the fact that these algebras has the same dimension,
since they are isotropy algebras of spaces of equal dimension.
\end{proof}

The equality of the isotropy Lie algebras $\mathfrak{l}=\mathfrak{z}_{\Theta
}$ show at once the equality of the isotropy subgroups if we know in advance
that they are connected, as happens for instance to complex Lie algebras.
The next statement shows that the isotropy groups indeed coincide.

\begin{proposition}
Let $L$ be the isotropy group of the action $a\colon G\times T^{\ast }%
\mathbb{F}_{\Theta }\rightarrow T^{\ast }\mathbb{F}_{\Theta }$ at $b_{\Theta
}$. Then, $L=Z_{\Theta }$.
\end{proposition}

\begin{proof}
By the previous lemma the Lie algebras of these groups coincide, therefore
their connected components of the identity $\left( Z_{\Theta }\right) _{0}$
and $L_{0}$ are equal. Since $L$ normalizes its Lie algebra, it follows that 
$L$ normalizes $\mathfrak{z}_{\Theta }$. Nevertheless, the normalizer of $%
\mathfrak{z}_{\Theta }$ is $Z_{\Theta }$. Therefore, $L\subset Z_{\Theta }$.

To verify the opposite inclusion, consider the restriction of the action $a$
to the subgroup $K$. For $A\in \mathfrak{k}$, $\theta \left( A\right)
=A^{\#} $. Thus, the action of $K$ on $T^{\ast }\mathbb{F}_{\Theta }$ is
linear. Therefore the isotropy group $K\cap L$ coincides with the isotropy
group of the action on $\mathbb{F}_{\Theta }$ at $b_{\Theta }$, that is, $%
K\cap L=K_{\Theta }$. Now, we know that $K_{\Theta }$ intercepts all
connected components of $Z_{\Theta }$. Therefore, the relations $K_{\Theta
}\subset L$ and $\left( Z_{\Theta }\right) _{0}=L_{0}$ imply that $Z_{\Theta
}\subset L$.
\end{proof}

\begin{remark}
The group $G$ that integrates the infinitesimal action $\theta $ is
necessarily the adjoint group $\mathrm{Aut}_{0}\left( \mathfrak{g}\right) $,
whose center is trivial. This happens because the action of $G$ on $T^{\ast }%
\mathbb{F}_{\Theta }$ is effective, as $G$ is a subgroup of diffeomorphisms
of $T^{\ast }\mathbb{F}_{\Theta }$. An effective action on the adjoint orbit
only happens for the adjoint group, since the centre $Z\left( G\right)
\subset Z_{\Theta }$ and if $z\in Z\left( G\right) $ then $\mathrm{Ad}\left(
z\right) =\mathrm{id}$.
\end{remark}

\subsection{Moment map on $T^{\ast }\mathbb{F}$}

The action $a\colon G\times T^{\ast }\mathbb{F}_{\Theta }\rightarrow T^{\ast
}\mathbb{F}_{\Theta }$ defined above is a Hamiltonian action, since $\theta
\left( Y\right) $ is a Hamiltonian field for each $Y\in \mathfrak{g}$. We
can then define a moment map $\mu \colon T^{\ast }\mathbb{F}_{\Theta
}\rightarrow \mathfrak{g}^{\ast }$, by $\mu \left( \xi \right) \left(
Y\right) =\mathrm{en}_{Y}\left( \xi \right) $, where $\mathrm{en}_{Y}\colon
T^{\ast }\mathbb{F}_{\Theta }\rightarrow \mathbb{R}$ is the energy function
of $\theta \left( Y\right) $ e $\xi \in T^{\ast }\mathbb{F}_{\Theta }$. That
is,

\begin{itemize}
\item if $A\in \mathfrak{k}$ then $\mu \left( \xi \right) \left( A\right)
=\xi \left( \widetilde{A}\left( x\right) \right) $, $x=\pi \left( \xi
\right) $, and

\item if $X\in \mathfrak{s}$ then $\mu \left( \xi \right) \left( X\right)
=\xi \left( \widetilde{X}\left( x\right) \right) +\langle X,x\rangle $, $%
x=\pi \left( \xi \right) $, where $\langle \cdot ,\cdot \rangle $ is
Cartan--Killing.
\end{itemize}

Associated with $\mu $ we define a cocycle $c\colon G\rightarrow \mathfrak{g}%
^{\ast }$ by 
\begin{equation*}
c\left( g\right) =\mu \left( g\cdot \xi \right) -\mathrm{Ad}^{\ast }\mu
\left( \xi \right) ,
\end{equation*}%
where $\xi \in T^{\ast }\mathbb{F}_{\Theta }$ is arbitrary since the second
hand side is constant as a function of $\xi $ (see \cite{abmars}). The map $%
c $ is a cocycle in the sense that 
\begin{equation*}
c\left( gh\right) =\mathrm{Ad}^{\ast }\left( g\right) c\left( h\right)
+c\left( g\right) ,
\end{equation*}%
which means that $c$ is a $1$-cocycle of the cohomology of the coadjoint
representation of $G$ on $\mathfrak{g}^{\ast }$.

In the case when $\mathfrak{g}$ is semisimple the Cartan--Killing form $%
\langle \cdot ,\cdot \rangle $ interchanges the representations: coadjoint $%
\mathrm{Ad}^{\ast }$ and adjoint $\mathrm{Ad}$. With this we can define a
moment map $\mu \colon T^{\ast }\mathbb{F}_{\Theta }\rightarrow \mathfrak{g}$
(same notation) by $\langle \mu \left( \xi \right) ,\cdot \rangle =\mathrm{en%
}_{Y}\left( \xi \right) $. So the cocycle becomes $c\left( g\right) =\mu
\left( g\cdot \xi \right) -\mathrm{Ad}\left( g\right) \mu \left( \xi \right) 
$, which satisfies $c\left( gh\right) =\mathrm{Ad}\left( g\right) c\left(
h\right) +c\left( g\right) $.

\begin{theorem}
\label{propmomento}Let $\mu :T^{\ast }\mathbb{F}_{\Theta }\mapsto \mathfrak{g%
}$ be the moment map of the action $a\colon G\times T^{\ast }\mathbb{F}%
_{\Theta }\rightarrow T^{\ast }\mathbb{F}_{\Theta }$ constructed above, and
let $c\colon G\rightarrow \mathfrak{g}$ be the corresponding cocycle. Then,

\begin{enumerate}
\item $c$ is identically zero, which means that $\mu \colon T^{\ast }\mathbb{%
F}_{\Theta }\rightarrow \mathfrak{g}$ is equivariant, that is, $\mu \left(
g\cdot \xi \right) =\mathrm{Ad}\mu \left( \xi \right) $.

\item $\mu $ is a diffeomorphism between $T^{\ast }\mathbb{F}_{\Theta }$ and
the adjoint orbit $\mathrm{Ad}\left( G\right) H_{\Theta }$.

\item $\mu ^{\ast }\omega =\Omega $, where $\Omega $ is the canonical
symplectic form of $T^{\ast }\mathbb{F}_{\Theta }$ and $\omega $ the (real)
Kirillov--Kostant--Souriaux form on $\mathrm{Ad}\left( G\right) H_{\Theta }$.

\item $\mu \colon T^{\ast }\mathbb{F}_{\Theta }\rightarrow \mathrm{Ad}\left(
G\right) H_{\Theta }$ is the inverse of the map $\iota \colon \mathrm{Ad}%
\left( G\right) H_{\Theta }\rightarrow T^{\ast }\mathbb{F}_{\Theta }$ of
Theorem \ref{teodifeocotan} given in (\ref{fordifeoexplicit}).
\end{enumerate}
\end{theorem}

\begin{proof}
The result is a consequence of the following items:

\begin{enumerate}
\item $\mu \left( b_{\Theta }\right) =H_{\Theta }$, where $b_{\Theta }$ is
the origin of $\mathbb{F}_{\Theta }$ also regarded as the null vector in $%
T_{b_{\Theta }}^{\ast }\mathbb{F}_{\Theta }$. In fact, if $A\in \mathfrak{k}$
then $\mu \left( b_{\Theta }\right) \left( A\right) =b_{\Theta }\left( 
\widetilde{A}\left( b_{\Theta }\right) \right) =0$. Whereas if $X\in 
\mathfrak{s}$ then 
\begin{eqnarray*}
\mu \left( b_{\Theta }\right) \left( X\right) &=&b_{\Theta }\left( 
\widetilde{X}\left( b_{\Theta }\right) \right) +\langle X,H_{\Theta }\rangle
\\
&=&\langle X,H_{\Theta }\rangle .
\end{eqnarray*}%
Therefore, $H_{\Theta }\in \mathfrak{g}$ satisfies $\mu \left( b_{\Theta
}\right) \left( Y\right) =\langle Y,H_{\Theta }\rangle $ for all $Y\in 
\mathfrak{g}$, which means that $\mu \left( b_{\Theta }\right) =H_{\Theta }$.

\item If $x\in \mathbb{F}_{\Theta }$ with $x=\mathrm{Ad}\left( k\right)
H_{\Theta }$, $k\in K$, then $\mu \left( x\right) =\mathrm{Ad}\left(
k\right) H_{\Theta }$. This follows by the same argument in the previous
item, where we regard $x$ as the zero vector in $T_{x}\mathbb{F}_{\Theta }$
and thus obtain $x\left( \widetilde{X}\left( x\right) \right) =0$ for any $%
X\in \mathfrak{g}$.

\item $c\left( k\right) =0$ if $k\in K$ as follows by definition $c\left(
k\right) =\mu \left( k\cdot b_{\Theta }\right) -\mathrm{Ad}\left( k\right)
\mu \left( b_{\Theta }\right) $ and the previous items.

\item $c\left( h\right) =0$ if $h\in A$. In fact, $\mathrm{Ad}\left(
h\right) \mu \left( b_{\Theta }\right) =\mathrm{Ad}\left( h\right) H_{\Theta
}=H_{\Theta }$. On the other hand, if $H\in \mathfrak{a}$ then $\theta
\left( H\right) \left( b_{\Theta }\right) =0$ since $H^{\#}\left( b_{\Theta
}\right) =0$ and $V_{H}\left( b_{\Theta }\right) =0$, since $\left(
df_{H}\right) _{b_{\Theta }}\left( \cdot \right) =\left( \widetilde{H}\left(
b_{\Theta }\right) ,\cdot \right) =0$. This implies that $b_{\Theta }$ is a
fixed point by the action of $A$ on $T^{\ast }\mathbb{F}_{\Theta }$.
Therefore, $\mu \left( h\cdot b_{\Theta }\right) =\mu \left( b_{\Theta
}\right) =H_{\Theta }$, concluding that $c\left( h\right) =\mu \left( h\cdot
b_{\Theta }\right) -\mathrm{Ad}\left( h\right) \mu \left( b_{\Theta }\right)
=0$.

\item $c\equiv 0$, that is, $\mu $ is equivariant: $\mu \left( g\cdot \xi
\right) =\mathrm{Ad}\mu \left( \xi \right) $. This follows from the polar
decomposition $G=K\left( \mathrm{cl}A^{+}\right) K$ and two applications of
the cocycle property. In fact, if $g=uhv\in K\left( \mathrm{cl}A^{+}\right)
K $ then 
\begin{eqnarray*}
c\left( g\right) &=&c\left( uhv\right) =\mathrm{Ad}\left( uh\right) c\left(
v\right) +c\left( uh\right) \\
&=&\mathrm{Ad}\left( u\right) c\left( h\right) +c\left( u\right) \\
&=&0.
\end{eqnarray*}

\item Since $\mu $ is equivariant and $\mu \left( b_{\Theta }\right)
=H_{\Theta }$, its image is contained in the adjoint orbit $\mathrm{Ad}%
\left( G\right) H_{\Theta }$. The diffeomorphism property is due to
equivariance, transitivity of $G$ on the spaces and the fact that the
isotropy subgroups on both spaces coincide. The pullback of item (3) is a
standard fact about moment maps of Hamiltonian actions.

\item To see the inverse of $\mu $ take $\xi =\iota \left( H_{\Theta
}+Z\right) \in T_{b_{\Theta }}^{\ast }\mathbb{F}_{\Theta }$. If $A\in 
\mathfrak{k}$ and $x\in \mathfrak{s}$ then 
\begin{equation*}
\langle \mu \left( \xi \right) ,A\rangle =\xi \left( \widetilde{A}\left(
b_{\Theta }\right) \right) \qquad \langle \mu \left( \xi \right) ,X\rangle
=\xi \left( \widetilde{X}\left( b_{\Theta }\right) \right) +f_{X}\left(
b_{\Theta }\right) .
\end{equation*}%
Write $A=A^{-}+A^{0}+A^{+}\in \mathfrak{g}=\mathfrak{n}_{\Theta }^{-}\oplus 
\mathfrak{z}_{\Theta }\oplus \mathfrak{n}_{\Theta }^{+}$. Then $\widetilde{A}%
\left( b_{\Theta }\right) =\widetilde{A^{-}}\left( b_{\Theta }\right) $ so $%
\xi \left( \widetilde{A}\left( b_{\Theta }\right) \right) =\langle
Z,A^{-}\rangle $. Since $\mathfrak{n}_{\Theta }^{+}$ is Cartan-Killing
orthogonal to $\mathfrak{z}_{\Theta }\oplus \mathfrak{n}_{\Theta }^{+}$ we
have $\xi \left( \widetilde{A}\left( b_{\Theta }\right) \right) =\langle
Z,A^{-}\rangle =\langle Z,A\rangle $, that is, 
\begin{equation*}
\langle \mu \left( \xi \right) ,A\rangle =\langle Z,A\rangle =\langle
H_{\Theta }+Z,A\rangle
\end{equation*}%
because $\langle H_{\Theta },A\rangle =0$. Similarly $\xi \left( 
\widetilde{X}\left( b_{\Theta }\right) \right) =\langle Z,X\rangle $ and
since $f_{X}\left( b_{\Theta }\right) =\langle H_{\Theta },X\rangle $ we
have $\langle \mu \left( \xi \right) ,X\rangle =\langle H_{\Theta
}+Z,X\rangle $. Hence $\mu \left( \iota \left( H_{\Theta }+Z\right) \right)
=H_{\Theta }+Z$, showing that $\mu $ and $\iota $ are inverse to each other.
\end{enumerate}
\end{proof}

\begin{remark}
(Other actions.) Besides the action defined above, there are other
infinitesimal actions $\mathfrak{g}$ on $T^{\ast }\mathbb{F}_{\Theta }$ that
play the same role:

\begin{enumerate}
\item Take $\theta ^{-}\left( A\right) =A^{\#}$ if $A\in \mathfrak{k}$ and $%
\theta ^{-}\left( X\right) =X^{\#}-V_{X}$ if $X\in \mathfrak{s}$. Then, $%
\theta ^{-}$ is still and infinitesimal representation, which gives rise to
another action of $G$.

\item If $\left( \left( \cdot ,\cdot \right) \right) $ is a $K$-invariant
Riemannian metric on $\mathbb{F}_{\Theta }$ such that each $\widetilde{X}$, $%
X\in \mathfrak{s}$, is a gradient of the function $\widehat{f}_{X}$ then the
same game can be played with the Hamiltonian vector field of $\widehat{F}%
_{X}=\widehat{f}\circ \pi $ in place of $V_{X}$.
\end{enumerate}
\end{remark}

\section{Embedding of adjoint orbits into products \label{secsplit}}

We recall here a known realization of the homogeneous space $G/Z_{\Theta }$
as an orbit in a product of flag manifolds (see \cite{orbadj}, Section 3,
for the details).

Let $w_{0}$ be the principal involution of the Weyl group $\mathcal{W}$,
that is, the element of highest length as a product of simple roots. Then $%
-w_{0}\mathfrak{a}^{+}=\mathfrak{a}^{+}$ and $-w_{0}\Sigma =\Sigma $, so
that $-w_{0}$ is a symmetry of the Dynkin diagram of $\Sigma $. For a subset 
$\Theta \subset \Sigma $ we put $\Theta ^{\ast }=-w_{0}\Theta $ and refer to 
$\mathbb{F}_{\Theta ^{\ast }}$ as the flag manifold dual to $\mathbb{F}%
_{\Theta }$. Clearly if $H_{\Theta }$ is a characteristic element for $%
\Theta $ then $-w_{0}H_{\Theta }$ is characteristic for $\Theta ^{\ast }$.
(Except for the simple Lie algebras with diagrams $A_{l}$, $D_{l}$ and $%
E_{6} $ all the flag manifolds are self-dual. In $A_{l}=\mathfrak{sl}\left(
n\right) $, $n=l+1$, we have for instance, the dual to the Grassmannian $%
\mathrm{Gr}_{k}\left( n\right) $ is $\mathrm{Gr}_{n-k}\left( n\right) $.)

Consider the diagonal action of $G$ on the product $\mathbb{F}_{\Theta
}\times \mathbb{F}_{\Theta ^{\ast }}$ as $(g,(x,y))\mapsto (gx,gy)$, $g\in G$%
, $x,y\in \mathbb{F}$. As we check next it has just one open and dense orbit
which is $G/Z_{\Theta }$.

Let $x_{0}$ be the origin of $\mathbb{F}_{\Theta }$. Since $G$ acts
transitively on $\mathbb{F}_{H}$, all the $G$-orbits of the diagonal action
have the form $G\cdot (x_{0},y)$, with $y\in \mathbb{F}_{\Theta ^{\ast }}$.
Thus, the $G$-orbits are in bijection with the orbits of the action of $%
P_{\Theta ^{\ast }}$ on $\mathbb{F}_{\Theta ^{\ast }}$, which is known to be
the orbits through $wy_{0}$, $w\in \mathcal{W}$, where $y_{0}$ is the origin
of $\mathbb{F}_{\Theta ^{\ast }}$. Hence the $G$-orbits are $G\cdot
(x_{0},wy_{0})$, $w\in \mathcal{W}$.

Now let $w_{0}$ be the principal involution of $\mathcal{W}$.

\begin{proposition}
The orbit $G\cdot (x_{0},\tilde{w}_{0}y_{0})$ is open and dense in $\mathbb{F%
}_{\Theta }\times \mathbb{F}_{\Theta ^{\ast }}$ and identifies to $G/Z_{H}$.
(Here and elsewhere $\tilde{w}$ stands for a representative in $K$ of $w\in 
\mathcal{W}$).
\end{proposition}

\begin{proof}
The isotropy subgroup at $(x_{0},\tilde{w}_{0}y_{0})$ is the intersection of
the isotropy subgroups at $x_{0}$ and $w_{0}y_{0}$. The first one is the
parabolic subgroup $P_{-H}$ associated to $\tilde{w}_{0}H^{\ast }=-H$, and
the second one is $P_{H}$, where $H$ is a characteristic element of $\Theta $.
Since $Z_{H}=P_{H}\cap P_{-H}$ the identification follows. Now the Lie
algebra $\mathfrak{z}_{H}=\mathfrak{p}_{H}\cap \mathfrak{p}_{-H}$ of $%
P_{H}\cap P_{-H}$ is complemented in $\mathfrak{g}$ by $\mathfrak{n}%
_{H}^{+}\cap \mathfrak{n}_{-H}^{+}$, with $\mathfrak{n}_{-H}^{+}=\sum_{%
\alpha \left( H\right) <0}\mathfrak{g}_{\alpha }$. Hence, the dimension of $%
G\cdot (x_{0},\tilde{w}_{0}y_{0})$ is the same as the dimension of $\mathbb{F%
}_{\Theta }\times \mathbb{F}_{\Theta ^{\ast }}$, so that the orbit is open.
An analogous reasoning shows that this is the only open orbit and hence
dense.
\end{proof}

In terms of this realization of $G/Z_{\Theta }$ as an open orbit, the map $%
G/Z_{\Theta }\rightarrow \mathbb{F}_{\Theta }$ is just the projection onto
the first factor. Also, if $\Theta \subset \Theta _{1}$ the projection $%
G/Z_{\Theta }\rightarrow G/Z_{\Theta _{1}}$ is inherited from the
projections $\mathbb{F}_{\Theta }\rightarrow \mathbb{F}_{\Theta _{1}}$ and $%
\mathbb{F}_{\Theta ^{\ast }}\rightarrow \mathbb{F}_{\Theta _{1}^{\ast }}$.

A flag manifold $\mathbb{F}_{\Theta }=G/P_{\Theta }$ is in bijection with
the set of parabolic subalgebras conjugate to $\mathfrak{p}_{\Theta }$,
since $P_{\Theta }$ is the normalizer of $\mathfrak{p}_{\Theta }$. From this
point of view the open orbit $G\cdot (x_{0},\tilde{w}_{0}y_{0})\subset 
\mathbb{F}_{\Theta }\times \mathbb{F}_{\Theta ^{\ast }}$ is characterized by
transversality: Two parabolic subalgebras $\mathfrak{p}_{1}\in \mathbb{F}%
_{\Theta }$ and $\mathfrak{p}_{2}\in \mathbb{F}_{\Theta ^{\ast }}$ are
transversal if $\mathfrak{g}=\mathfrak{p}_{1}+\mathfrak{p}_{2}$, or
equivalently if $\mathfrak{n}\left( \mathfrak{p}_{1}\right) \cap \mathfrak{p}%
_{2}=\mathfrak{p}_{1}\cap \mathfrak{n}\left( \mathfrak{p}_{2}\right) =\{0\}$%
, where $\mathfrak{n}\left( \cdot \right) $ stands for the nilradical (see 
\cite{smmax}). Then the open orbit $G\cdot (x_{0},\tilde{w}_{0}y_{0})$ is
the set of pairs of transversal subalgebras. In particular, the set of
subalgebras transversal to the origin $x_{0}\in \mathbb{F}_{\Theta }$ is the
open cell $N^{+}\tilde{w}_{0}y_{0}$ with $y_{0}$ the origin of $\mathbb{F}%
_{\Theta ^{\ast }}$. More generally the set of subalgebras transversal to $%
gx_{0}$, $g\in G$, is the open cell $gN^{+}\tilde{w}_{0}x_{0}$.

The fixed points of a split-regular element $h\in A^{+}=\exp \mathfrak{a}%
^{+} $ in a flag manifold $\mathbb{F}_{\Theta }$ are isolated. The set of
fixed points is the orbit through the origin of $\mathrm{Norm}_{K}\left( 
\mathfrak{a}\right) $ and factors down to the Weyl group $\mathcal{W}=%
\mathrm{Norm}_{K}\left( \mathfrak{a}\right) /\mathrm{Cent}_{K}\left( 
\mathfrak{a}\right) $.

\section{Adjoint orbits and representations of $\mathfrak{g}$ \label%
{secrepre}}

In this section we give realizations of the coset spaces $G/Z_{\Theta }$
based on representations of $\mathfrak{g}$. It will be convenient to assume
that $\mathfrak{g}$ is a complex algebra, even though the theory works, with
some adaptations, for real algebras. This new description helps to establish
a bridge between the adjoint orbit and the open orbit in the product.

\subsection{The adjoint action of $G$ on $\mathrm{End}(V)$}

Let $\mathfrak{h}$ be a Cartan subalgebra of $\mathfrak{g}$ and denote by $%
\mathfrak{h}_{\mathbb{R}}$ the real subspace of $\mathfrak{h}$ spanned by $%
H_{\alpha }$, $\alpha \in \Pi $, where $\alpha \left( \cdot \right) =\langle
H_{\alpha },\cdot \rangle $. Fix a Weyl chamber $\mathfrak{h}_{\mathbb{R}%
}^{+}$ and let $\Sigma =\{\alpha _{1},\ldots ,\alpha _{l}\}$ be the
corresponding system of simple roots. The fundamental weights $\{\mu
_{1},\ldots ,\mu _{l}\}$ are defined by 
\begin{equation*}
\langle \alpha _{i}^{\vee },\mu _{j}\rangle =\frac{2\langle \alpha _{i},\mu
_{j}\rangle }{\langle \alpha _{i},\alpha _{i}\rangle }=\delta _{ij}.
\end{equation*}%
If $\mu =a_{1}\mu _{1}+\cdots +a_{l}\mu _{l}$ with $a_{i}\in \mathbb{N}$,
then there exists a unique irreducible representation $\rho _{\mu }$ of $%
\mathfrak{g}$ with highest weight $\mu $. If $V=V\left( \mu \right) $ is the
representation space, then $V$ decomposes into weight spaces (the
simultaneous eigenspaces for $\rho _{\mu }\left( H\right) $, $H\in \mathfrak{%
h}$)%
\begin{equation*}
V=\sum_{\nu }V_{\nu }.
\end{equation*}%
The highest weight space $V_{\mu }$ has dimension $1$ and is characterized
by the fact that $\rho _{\mu }\left( X\right) v=0$ if $v\in V_{\mu }$ and $%
X\in \sum_{\alpha >0}\mathfrak{g}_{\alpha }$. The remaining weights take the
form $\nu =\mu -\left( n_{1}\alpha _{1}+\cdots +n_{l}\alpha _{l}\right) $
with $n_{i}\in \mathbb{N}$. Thus, if $H\in \mathfrak{h}_{\mathbb{R}}^{+}$
then $\mu \left( H\right) $ is the largest eigenvalue of $\rho _{\mu }\left(
H\right) $. The set of weights of the representation is invariant by the
Weyl group. If $w_{0}$ is the main involution, then $w_{0}\mu $ is a lowest
weight, that is, $\left( w_{0}\mu \right) \left( H\right) =\mu \left(
w_{0}H\right) $ is the smallest eigenvalue of $\rho _{\mu }\left( H\right) $
if $H\in \mathfrak{h}_{\mathbb{R}}^{+}$.

If $K\subset G$ is the maximal compact subgroup, then $V$ can be endowed
with a $K$-invariant Hermitian form $\left( \cdot ,\cdot \right) ^{\mu }$
such that the weight spaces are pairwise orthogonal. Such a Hermitian form
is unique up to scale, because the representation of $K$ on $V$ is
irreducible.

In section \ref{compact} we study Lagrangean submanifolds of adjoint orbits $%
\mathrm{Ad}\left( G\right) H_{0}$ with $H_{0}\in \mathrm{cl}\left( \mathfrak{%
h}_{\mathbb{R}}^{+}\right) $, embedded into products $\mathbb{F}%
_{H_{0}}\times \mathbb{F}_{H_{0}^\ast}$. There is a freedom of choice to
pick the element $H_{0} $, producing the same flag $\mathbb{F}_{H_{0}}$. In
what follows we will choose a convenient $H_{0}$.

Let $\Theta _{0}=\Theta \left( H_{0}\right) =\{\alpha \in \Sigma :\alpha
\left( H_{0}\right) =0\}$, that is, $H_{0}$ is characteristic for $\Theta
_{0}$. Let $\mu $ be a highest weight such that, for $\alpha \in \Sigma $, $%
\langle \alpha ^{\vee },\mu \rangle =0$ if and only if $\alpha \in \Theta
_{0}$. (For example, $\mu =\mu _{i_{1}}+\cdots +\mu _{is}$ if $\Sigma
\setminus \Theta _{0}=\{\alpha _{i_{1}},\ldots ,\alpha _{i_{s}}\}$.) Define $%
H_{\mu }\in \mathfrak{h}_{\mathbb{R}}$ by $\mu \left( \cdot \right) =\langle
H_{\mu }\,\cdot \rangle $. Then, the centralizers of $H_{\mu }$ and $H_{0}$
coincide, since $\Theta _{0}$ is the set of simple roots that vanish on $%
H_{0}$ as well as on $H_{\mu }$. Hence the adjoint orbits $\mathrm{Ad}\left(
G\right) H_{\mu }$ and $\mathrm{Ad}\left( G\right) H_{0}$ give rise to the
same homogeneous space $G/Z_{H_{0}}=G/Z_{H_{\mu }}$ and the flags $\mathbb{F}%
_{H_{0}}$ and $\mathbb{F}_{H_{\mu }}$ coincide.

From now on we take $H_{0}=H_{\mu }$ with $\mu $ a highest weight, $\mu =\mu
_{i_{1}}+\cdots +\mu _{is}$.

Let $G$ be the linear connected group with Lie algebra $\rho _{\mu }\left( 
\mathfrak{g}\right) \approx \mathfrak{g}$ and consider its action on the
projetive space $\mathbb{P}\left( V\right) $ of the representation space $%
V=V\left( \mu \right) $. It is well known that this choice of $\mu $
guaranties that the projective orbit of $G$ by the subspace of highest
weight $V_{\mu }\in \mathbb{P}\left( V\right) $ is the flag $\mathbb{F}%
_{H_{\mu }}=\mathbb{F}_{\Theta _{0}}$.

Consider now the dual representations $\rho _{\mu }^{\ast }$ of $\mathfrak{g}
$ and $G$ on $V^{\ast }$ as $\rho _{\mu }^{\ast }\left( X\right) \left(
\varepsilon \right) =-\varepsilon \circ \rho _{\mu }\left( X\right) $ and $%
\rho _{\mu }^{\ast }\left( g\right) \left( \varepsilon \right) =\varepsilon
\circ \rho _{\mu }\left( g^{-1}\right) $ if $\varepsilon \in V^{\ast }$, $%
X\in \mathfrak{g}$ and $g\in G$. Choose a basis $\{v_{0},\ldots ,v_{N}\}$ de 
$V$ adapted to the decomposition in weight spaces with $v_{0}\in V_{\mu }$.
Denote by $\{\varepsilon _{0},\ldots ,\varepsilon _{N}\}$ the dual basis $%
\varepsilon _{i}\left( v_{j}\right) =\delta _{ij}$. Then $\varepsilon _{0}$
generates a subspace of \textquotedblleft lowest \textquotedblright\ weight
of $V^{\ast }$, in the sense that

\begin{enumerate}
\item $\rho _{\mu }^{\ast }\left( H\right) \left( \varepsilon _{0}\right)
=-\mu \left( H\right) \varepsilon _{0}$ if $H\in \mathfrak{h}$. Indeed, if
the basis element $v_{i}\in V_{\nu }$, then 
\begin{equation*}
\rho _{\mu }^{\ast }\left( H\right) \left( \varepsilon _{0}\right) \left(
v_{i}\right) =-\varepsilon _{0}\circ \rho _{\mu }\left( H\right) \left(
v_{i}\right) =-\nu \left( H\right) \varepsilon _{0}\left( v_{i}\right) =-\nu
\left( H\right) \delta _{0i}.
\end{equation*}

\item $\rho _{\mu }^{\ast }\left( X\right) \left( \varepsilon _{0}\right) =0$
if $X\in \sum_{\alpha <0}\mathfrak{g}_{\alpha }$, since $\rho _{\mu }^{\ast
}\left( H\right) \left( \varepsilon _{0}\right) \left( v_{i}\right)
=-\varepsilon _{0}\left( \rho _{\mu }\left( X\right) \right) v_{i}$ and $%
\rho _{\mu }\left( X\right) $ takes a weight space $V_{\nu }$ to the sum of
spaces of weights smaller than $\nu $.
\end{enumerate}

Therefore, $-\mu $ is the lowest weight of $V^{\ast }$. So, the highest
weight is $\mu ^{\ast }=-w_{0}\mu $. This means that the projective orbit of
the highest weight (and of $\varepsilon _{0}$) on $V^{\ast }$ is the dual
flag $\mathbb{F}_{H_{\mu }^{\ast }}$.

\begin{example}
If $\mathfrak{g}=\mathfrak{sl}\left( n,\mathbb{C}\right) $ then the
fundamental weights are $\lambda _{1}$, $\lambda _{1}+\lambda _{2}$, \ldots ,%
$\lambda _{1}+\cdots +\lambda _{n-1}$, where $\lambda _{i}$ is the
functional that associates the $i$-th eigenvalue of the diagonal matrix $%
H\in \mathfrak{h}$. If $\mu $ is a fundamental weight $\mu =\lambda
_{1}+\cdots +\lambda _{k}$ then the irreducible representation with highest
weight $\mu $ is the representation of $\mathfrak{g}$ on the $k$-th exterior
power $\Lambda ^{k}\mathbb{C}^{n}$ of $\mathbb{C}^{n}$. The highest weight
space is generated by $e_{1}\wedge \cdots \wedge e_{k}$ ($e_{i}$ are the
basis vectors of $\mathbb{C}^{n}$). The $G$-orbit of $e_{1}\wedge \cdots
\wedge e_{k}$ is the set of decomposable elements of $\Lambda ^{k}\mathbb{C}%
^{n}$, so the projective $G$-orbit is identified to the Grassmannian $%
\mathrm{Gr}_{k}\left( n\right) $. The dual flag of $\mathrm{Gr}_{k}\left(
n\right) $ is $\mathrm{Gr}_{n-k}\left( n\right) $ which is the projective
orbit on $\Lambda ^{n-k}\mathbb{C}^{n}$, identified to the dual $\Lambda ^{k}%
\mathbb{C}^{n}$ by a choice of volume form on $\mathbb{C}^{n}$. The lowest
weight space on $\Lambda ^{n-k}\mathbb{C}^{n}$ is generated by $%
e_{k+1}\wedge \cdots \wedge e_{n}$.
\end{example}

Keeping the same highest weight $\mu $, consider the tensor product $%
V\otimes V^{\ast }$ which is isomorphic to the space of endomorphisms $%
\mathrm{End}\left( V\right) $ of $V$. 
The group $G$ gets represented on $V\otimes V^{\ast }$ by $g\cdot \left(
v\otimes \varepsilon \right) =\rho _{\mu }\left( g\right) v\otimes \rho
_{\alpha }^{\ast }\left( g\right) \varepsilon $, which is isomorphic to the
adjoint representation of $G$ on $\mathrm{End}\left( V\right) $.

Once again, let $v_{0}$ and $\varepsilon _{0}$ be generators of the spaces
of highest weight of $V$ and lowest of $V^{\ast }$, respectively. With this
notation, our fourth model of the adjoint orbit is the $G$-orbit of $%
v_{0}\otimes \varepsilon _{0}$. To prove that this orbit is indeed $%
G/Z_{H_{0}}$ we shall consider the moment map of the representation. Namely,
the map $\overline{M}\colon V\otimes V^{\ast }\rightarrow \mathfrak{g}^{\ast
}$ defined by 
\begin{equation*}
\overline{M}\left( v\otimes \varepsilon \right) \left( Z\right) =\varepsilon
\left( \rho _{\mu }\left( Z\right) v\right) \qquad v\in V,~\varepsilon \in
V^{\ast },~Z\in \mathfrak{g}.
\end{equation*}%
Since $\mathfrak{g}$ is semi-simple, $\mathfrak{g}\approx \mathfrak{g}^{\ast
}$ via the Cartan-Killing form $\langle \cdot ,\cdot \rangle $ and we can
take the moment map $M\colon V\otimes V^{\ast }\rightarrow \mathfrak{g}$
given by 
\begin{equation*}
\langle M\left( v\otimes \varepsilon \right) ,Z\rangle =\varepsilon \left(
\rho _{\mu }\left( Z\right) v\right) \qquad v\in V,~\varepsilon \in V^{\ast
},~Z\in \mathfrak{g}.
\end{equation*}%
It is well known and easy to prove that $M$ is equivariant with respect to
the representations on $V\otimes V^{\ast }$ and $\mathfrak{g}$. In fact,
since $\rho _{\mu }\left( \mathrm{Ad}\left( g\right) Z\right) =\rho _{\mu
}\left( g\right) \rho _{\mu }\left( Z\right) \rho _{\mu }\left(
g^{-1}\right) $ we have 
\begin{eqnarray*}
\langle \mathrm{Ad}\left( g\right) M\left( v\otimes \varepsilon \right)
,Z\rangle &=&\langle \mathrm{Ad}\left( g\right) M\left( v\otimes \varepsilon
\right) ,\mathrm{Ad}\left( g^{-1}\right) Z\rangle \\
&=&\varepsilon \left( \rho _{\mu }\left( g^{-1}\right) \rho _{\mu }\left(
Z\right) \rho _{\mu }\left( g\right) v\right) \\
&=&\rho _{\mu }\left( g\right) v\otimes \rho _{\mu }^{\ast }\left( g\right)
\varepsilon =g\cdot \left( v\otimes \varepsilon \right) .
\end{eqnarray*}%
The same calculation shows that $\overline{M}$ is equivariant with respect
to the co-adjoint representation.

In the semi-simple case the moment map has the following geometric
interpretation: $\rho _{\mu }$ is a faithful representation, thus $\mathfrak{%
g}\approx \rho _{\mu }\left( \mathfrak{g}\right) \subset \mathrm{End}\left(
V\right) $. The trace form $\mathrm{tr}\left( AB\right) $ on $\mathrm{End}%
\left( V\right) $ is non-degenerate. Then the moment map is just the
orthogonal projection with respect to the trace form of $\mathrm{End}\left(
V\right) \approx V\otimes V^{\ast }$ onto $\rho _{\mu }\left( \mathfrak{g}%
\right) \approx \mathfrak{g}$.

As a consequence of equivariance, it follows that the image of a $G $-orbit
on $V\otimes V^{\ast }$ by $M$ is an adjoint orbit.

\begin{lemma}
The image of the $G$-orbit $G\cdot \left( v_{0}\otimes \varepsilon
_{0}\right) $ by $M$ is the adjoint orbit of $H_{\mu }$ defined by $\mu
\left( \cdot \right) =\langle H_{\mu },\cdot \rangle $. (That is, the image
by $\overline{M}$ of $G\cdot \left( v_{0}\otimes \varepsilon _{0}\right) $
is the linear functional on $\mathfrak{g}^{\ast }$ that coincides with $\mu $
on $\mathfrak{h}$ and vanishes on the sum of root spaces.)
\end{lemma}

\begin{proof}
If $\alpha $ is a root and $X\in \mathfrak{g}_{\alpha }$ then 
\begin{equation*}
\varepsilon _{0}\left( \rho _{\mu }\left( X\right) v_{0}\right) =\left( \rho
_{\mu }\left( X\right) v_{0}\right) \otimes \varepsilon _{0}=-v_{0}\otimes
\left( \rho _{\mu }^{\ast }\left( X\right) \varepsilon _{0}\right) .
\end{equation*}%
The second term vanishes if $\alpha >0$ whereas if $\alpha <0$ the third
term vanishes. Hence $\langle M\left( \varepsilon _{0}\otimes v_{0}\right)
,X\rangle =0$. But, if $H\in \mathfrak{h}$ then 
\begin{equation*}
\varepsilon _{0}\left( \rho _{\mu }\left( H\right) v_{0}\right) =\mu \left(
H\right) \varepsilon _{0}\left( v_{0}\right) =\mu \left( H\right) ,
\end{equation*}%
that is, $\langle M\left( \varepsilon _{0}\otimes v_{0}\right) ,H\rangle
=\mu \left( H\right) $ which shows that $M\left( \varepsilon _{0}\otimes
v_{0}\right) =H_{\mu }$. Consequently, \newline
$M\left( G\cdot \left( v_{0}\otimes \varepsilon _{0}\right) \right) =\mathrm{%
Ad}\left( G\right) H_{\mu }$.
\end{proof}

\begin{proposition}
The $G$-orbit $G\cdot \left( v_{0}\otimes \varepsilon _{0}\right) $ is the
homogeneous space $G/Z_{H_{\mu }}$.
\end{proposition}

\begin{proof}
Set $G\cdot \left( v_{0}\otimes \varepsilon _{0}\right) =G/L$. We want to
show that $L=Z_{H_{\mu }}$. The equivariance of $M$ together with the
equality $M\left( G\cdot \left( v_{0}\otimes \varepsilon _{0}\right) \right)
=\mathrm{Ad}\left( G\right) H_{\mu }$ imply that the isotropy subgroup at $%
v_{0}\otimes \varepsilon _{0}$ is contained in the isotropy subgroup at $%
H_{\mu }$, that is, $L\subset Z_{H_{\mu }}$. Since $Z_{H_{\mu }}$ is
connected, to show the opposite inclusion it suffices to show that the Lie
algebra $\mathfrak{z}_{H_{\mu }}$ of $Z_{H_{\mu }}$ is contained in the
isotropy algebra of $v_{0}\otimes \varepsilon _{0}$.

To verify this, we observe that the isotropy algebra of $v_{0}$ is $\ker \mu
+\sum_{\alpha >0}\mathfrak{g}_{\alpha }+\sum_{\alpha \in \langle \Theta
_{0}\rangle ^{-}}\mathfrak{g}_{\alpha }$, where $\langle \Theta _{0}\rangle
^{-}$ is the set of negative roots generated by $\Theta _{0}$, which in turn
is the set of simple roots that vanish on $H_{0}$ (or $H_{\mu }$). In this
sum, the first term is given by elements $H\in \mathfrak{h}$ such that $\rho
_{\mu }\left( H\right) v_{0}=0$. The second term appears in the isotropy
algebra because $v_{0}$ is a highest weight vector. Finally the last term
comes from the fact that if $\alpha $ is a negative root and $X\in \mathfrak{%
g}_{\alpha }$, then $\rho _{\mu }\left( X\right) v_{0}=0$ if and only if $%
\langle \alpha ^{\vee },\mu \rangle =0$. The roots that satisfy this
equality are precisely the roots in $\langle \Theta _{0}\rangle ^{-}$.
Analogously, the isotropy algebra at $\varepsilon _{0}$ is given by $\ker
\mu +\sum_{\alpha <0}\mathfrak{g}_{\alpha }+\sum_{\alpha \in \langle \Theta
_{0}\rangle ^{+}}\mathfrak{g}_{\alpha }$ where $\langle \Theta _{0}\rangle
^{+}$ is the set of positive roots generated by $\Theta _{0}$.

Now, set $X\in \mathfrak{z}_{H_{\mu }}=\mathfrak{h}\oplus \sum_{\alpha \in
\langle \Theta _{0}\rangle ^{\pm }}\mathfrak{g}_{\alpha }$. If $X\in
\sum_{\alpha \in \langle \Theta _{0}\rangle ^{\pm }}\mathfrak{g}_{\alpha }$
then $\rho _{\mu }\left( X\right) v_{0}\otimes \varepsilon _{0}+v_{0}\otimes
\rho _{\mu }^{\ast }\left( X\right) \varepsilon _{0}=0$ since $X$ belongs to
the isotropy algebras of $v_{0}$ and $\varepsilon _{0}$. Whereas if $H\in 
\mathfrak{h}$ then 
\begin{equation*}
\rho _{\mu }\left( H\right) v_{0}\otimes \varepsilon _{0}+v_{0}\otimes \rho
_{\mu }^{\ast }\left( H\right) \varepsilon _{0}=\mu \left( H\right)
v_{0}\otimes \varepsilon _{0}-\mu \left( H\right) v_{0}\otimes \varepsilon
_{0}=0.
\end{equation*}%
Therefore, $\mathfrak{z}_{H_{\mu }}$ is contained in the isotropy subalgebra
of $v_{0}\otimes \varepsilon _{0}$.
\end{proof}

\begin{corollary}
The restriction of the moment map defines a diffeomorphism $M\colon G\cdot
\left( v_{0}\otimes \varepsilon _{0}\right) \rightarrow \mathrm{Ad}\left(
G\right) H_{\mu }$.
\end{corollary}

Via this diffeomorphism, the height function $f_{H}\colon \mathrm{Ad}\left(
G\right) H_{\mu }\rightarrow \mathbb{C}$ defines a function, also denoted by 
$f_{H}$, on the orbit $G\cdot \left( v_{0}\otimes \varepsilon _{0}\right) $.
This function has a simple expression.

\begin{proposition}
Let $v\otimes \varepsilon \in G\cdot \left( v_{0}\otimes \varepsilon
_{0}\right) $. Then $f_{H}\left( v\otimes \varepsilon \right) =\varepsilon
\left( \rho _{\mu }\left( H\right) v\right) =\mathrm{tr}\left( \left(
v\otimes \varepsilon \right) \rho _{\mu }\left( H\right) \right) $.
\end{proposition}

\begin{proof}
For a moment, denote by $\widetilde{f}_{H}$ the function $f_{H}$ defined on $%
G\cdot \left( v_{0}\otimes \varepsilon _{0}\right) $. Then 
\begin{equation*}
\widetilde{f}_{H}\left( v\otimes \varepsilon \right) =f_{H}\left( M\left(
v\otimes \varepsilon \right) \right) =\langle M\left( v\otimes \varepsilon
\right) ,H\rangle ,
\end{equation*}%
which is $\varepsilon \left( \rho _{\mu }\left( H\right) v\right) $ by the
definition of $M$. In the expression involving the trace, $v\otimes
\varepsilon $ is regarded as an element of $\mathrm{End}\left( V\right) $
and the second equality follows from $\varepsilon \left( Sv\right) =\mathrm{%
tr}\left( \left( v\otimes \varepsilon \right) S\right) $ which holds for any 
$S\in \mathrm{End}\left( V\right) $.
\end{proof}

\begin{remark}
In the above statement it is subsumed that all elements of the orbit $G\cdot
\left( v_{0}\otimes \varepsilon _{0}\right) $ are decomposable, that is,
have the form $v\otimes \varepsilon $. This happens because all elements of
the orbit have the form $\left( \rho _{\mu }\left( g\right) v_{0}\right)
\otimes \left( \rho _{\mu }^{\ast }\left( g\right) \varepsilon _{0}\right) $.
\end{remark}

\begin{remark}
The isomorphism between $V\otimes V^{\ast }$ and $\mathrm{End}\left(
V\right) $ associates to elements of the orbit $G\cdot \left( v_{0}\otimes
\varepsilon _{0}\right) $ linear transformations of rank $1$ with
transversal kernel and image.
\end{remark}

\subsection{Isomorphism with the open orbit in $\mathbb{F}_{H_{\protect\mu %
}}\times \mathbb{F}_{H_{\protect\mu }^{\ast }}$}

As mentioned earlier, the flags $\mathbb{F}_{H_{\mu }}$ and $\mathbb{F}%
_{H_{\mu }^{\ast }}$ are obtained as projective orbits in $\mathbb{P}\left(
V\right) $ and $\mathbb{P}\left( V^{\ast }\right) $, respectively. In the
identification of these flags with the projective orbits, the origin of $%
\mathbb{F}_{H_{\mu }}$ is identified with the highest weight space $\left[
v_{0}\right] $. This happens because the isotropy algebra of $\left[ v_{0}%
\right] $ contains $\sum_{\alpha >0}\mathfrak{g}_{\alpha }$. In the
identification of $\mathbb{F}_{H_{\mu }}$ with the adjoint orbit of the
compact group $K$, this origin is precisely $H_{\mu }$.

On the other hand, $\left[ \varepsilon _{0}\right] $ is the lowest weight
space in $V^{\ast }$. The isotropy algebra at $\left[ \varepsilon _{0}\right]
$ contains $\sum_{\alpha <0}\mathfrak{g}_{\alpha }$. This way, $\left[
\varepsilon _{0}\right] \in \mathbb{P}\left( V^{\ast }\right) $ is
identified with $w_{0}b^{\ast }\in \mathbb{F}_{H_{\mu }^{\ast }}$, where $%
b^{\ast }$ is the origin of $\mathbb{F}_{H_{\mu }^{\ast }}$. Under the
identification of $\mathbb{F}_{H_{\mu }^{\ast }}$ with the adjoint orbit of
the compact group $K$, the origin is $-H_{\mu }=w_{0}H_{\mu }^{\ast }$.

We use these identifications to see $\mathbb{F}_{H_{0}}\times \mathbb{F}%
_{H_{0}^{\ast }}$ as the product of the projective orbits $G\cdot \left[
v_{0}\right] \times G\cdot \left[ \varepsilon _{0}\right] \subset \mathbb{P}%
\left( V\right) \times \mathbb{P}\left( V^{\ast }\right) $. Then, the open
orbit in $\mathbb{F}_{H_{0}}\times \mathbb{F}_{H_{0}^{\ast }}$ becomes the
diagonal $G$-orbit of $\left( \left[ v_{0}\right] ,\left[ \varepsilon _{0}%
\right] \right) \in \mathbb{P}\left( V\right) \times \mathbb{P}\left(
V^{\ast }\right) $. Denote this orbit by $G\cdot \left( \left[ v_{0}\right] ,%
\left[ \varepsilon _{0}\right] \right) $, that is, 
\begin{equation*}
G\cdot \left( \left[ v_{0}\right] ,\left[ \varepsilon _{0}\right] \right)
=\{\left( \rho _{\mu }\left( g\right) \left[ v_{0}\right] ,\rho _{\mu
}^{\ast }\left( g\right) \left[ \varepsilon _{0}\right] \right) \in \mathbb{P%
}\left( V\right) \times \mathbb{P}\left( V^{\ast }\right) :g\in G\}.
\end{equation*}

We can now describe the diffeomorphism between the orbit $G\cdot \left(
v_{0}\otimes \varepsilon _{0}\right) \subset V\otimes V^{\ast }$ and the
orbit $G\cdot \left( \left[ v_{0}\right] ,\left[ \varepsilon _{0}\right]
\right) \subset \mathbb{F}_{H_{\mu }}\times \mathbb{F}_{H_{\mu }^{\ast
}}\subset \mathbb{P}\left( V\right) \times \mathbb{P}\left( V^{\ast }\right) 
$. In fact, the diffeomorphism associates $g\cdot \left( \left[ v_{0}\right]
,\left[ \varepsilon _{0}\right] \right) =\left( \rho _{\mu }\left( g\right) %
\left[ v_{0}\right] ,\rho _{\mu }^{\ast }\left( g\right) \left[ \varepsilon
_{0}\right] \right) $ to $g\cdot \left( v_{0}\otimes \varepsilon _{0}\right)
=\rho _{\mu }\left( g\right) v_{0}\otimes \rho _{\mu }^{\ast }\left(
g\right) \varepsilon _{0}$. We obtain,

\begin{proposition}
\label{phi} Let $\Phi :G\cdot \left( v_{0}\otimes \varepsilon _{0}\right)
\rightarrow G\cdot \left( \left[ v_{0}\right] ,\left[ \varepsilon _{0}\right]
\right) $ be the diffeomorphism obtained by identification of both orbits
with $G/Z_{H_{\mu }}$. If $v\otimes \varepsilon \in G\cdot \left(
v_{0}\otimes \varepsilon _{0}\right) $ then $\Phi \left( v\otimes
\varepsilon \right) =\left( \left[ v\right] ,\left[ \varepsilon \right]
\right) $ with inverse $\Phi ^{-1}\left( \left[ v\right] ,\left[ \varepsilon %
\right] \right) =\left( v\otimes \varepsilon \right) $.
\end{proposition}

\begin{proof}
Our previous argument already proved this. Nevertheless, it is worth
observing that the maps $v\otimes \varepsilon \mapsto \left( \left[ v\right]
,\left[ \varepsilon \right] \right) $ and $\left( \left[ v\right] ,\left[
\varepsilon \right] \right) \mapsto v\otimes \varepsilon $ are well defined,
since $v_{1}\otimes \varepsilon _{1}=v\otimes \varepsilon $ is equivalent to 
$v_{1}=av$ and $\varepsilon _{1}=a^{-1}\varepsilon $ which is also
equivalent to $\left( \left[ v_{1}\right] ,\left[ \varepsilon _{1}\right]
\right) =\left( \left[ v\right] ,\left[ \varepsilon \right] \right) $.
\end{proof}


\subsection{Isomorphism with $T^{\ast }\mathbb{F}_{H_{\protect\mu }}$}

First of all, we recall the isomorphism between the adjoint orbit $\mathcal{O%
}(H_\mu)$ and the cotangent bundle $T^\ast \mathbb{F}_\Theta$. Here, $H_\mu$
remains fixed as in the previous sections and is characteristic for $\Theta$%
, that is, $\Theta=\{\alpha\in\Sigma: \alpha(H_\mu)=0\}$.

By the Iwasawa decomposition $G=KAN$ we can write $G=KP_\Theta$ and the
adjoint action of $P_\Theta$ on $H_\mu$ is given by $\mathrm{Ad}
(P_\theta)\cdot H_\mu=H_\mu+\mathfrak{n}_\Theta^+$, where $\mathfrak{n}%
_\Theta^+ =\sum_{\Pi^+\setminus \langle \Theta\rangle}{\mathfrak{g}_\alpha}$%
. Thus, 
\begin{equation*}
\mathcal{O}(H{_\mu})=\mathrm{Ad}(G)H_\mu=\mathrm{Ad}(K)(H_\mu+\mathfrak{n}%
_\Theta^+)=\bigcup_{k\in K}\mathrm{Ad}(k)(H_\mu+\mathfrak{n}_\Theta^+).
\end{equation*}

The identification of the adjoint orbit with the cotangent bundle is given
by the map that associates to each element of the adjoint orbit $\mathrm{Ad}%
(k)(H_\mu+X)$, $X\in \mathfrak{n}_\Theta^+$ the linear functional $f\in
(T_{kb_0}\mathbb{F}_\Theta)^\ast$ given by $f(Y)=\langle \mathrm{Ad}
(k)X,Y\rangle$, $Y\in T_{kb_0}\mathbb{F}_\Theta$.


\begin{proposition}
Let $\mu$ be a highest weight and $v_0,\varepsilon_0$ the generators of the
highest weight space on $V$ and lowest weight space on $V^\ast$
respectively. The diffeomorphism between $G\cdot(v_0\otimes \varepsilon_0)$
and $T^\ast\mathbb{F}_\Theta$ is given by 
\begin{equation}
g\cdot (v_0\otimes \varepsilon_0)\mapsto (Y \mapsto \langle \mathrm{Ad}(k)X,
Y \rangle, Y \in T_{kb_0}\mathbb{F}_\Theta),
\end{equation}
where $g=kp$ is the Iwasawa decomposition, $\mathrm{Ad}(p)H_\mu=H_\mu+X$,
and the flag $\mathbb{F}_\Theta$ is determined by $H_\mu$.
\end{proposition}

\section{Compactified adjoint orbits}

\label{compact} We compactify adjoint orbits $\mathcal{O}\left( H_{0}\right) 
$ to products of flags $\mathbb{F}_{H_{0}}\times \mathbb{F}_{H_{0}^{\ast }}$
as an auxiliary tool to identify Lagrangean submanifolds of the orbits. We
choose canonical complex structures on $\mathbb{F}_{H_{0}}$ and $\mathbb{F}%
_{H_{0}^{\ast }}$ so that, for an element $w_{0}$ of the Weyl group $%
\mathcal{W}$, the right action $R_{w_{0}}\colon \mathbb{F}%
_{H_{0}}\rightarrow \mathbb{F}_{H_{0}^{\ast }}$ is anti-holomorphic
(proposition \ref{ah}). Consequently the map $R_{w_{0}}\colon \mathbb{F}%
_{H_{0}}\rightarrow \mathbb{F}_{H_{0}^{\ast }}$ is anti-symplectic with
respect to the K\"{a}hler forms on $\mathbb{F}_{H_{0}}$ and $\mathbb{F}%
_{H_{0}^{\ast }}$ given by the Borel metric and canonical complex structures
(corollary \ref{corantisimpl}). We then obtain further examples of
Lagrangean graphs by composites (either on the left or on the right) of $%
R_{w_{0}}$ with symplectic maps.

\subsection{Lagrangean graphs in adjoint orbits}

On one hand, $\mathcal{O}\left( H_{0}\right) $ can be embedded as an open
dense submanifold in a product of two flags (section \ref{secsplit}); on the
other hand, graphs of symplectic maps are Lagrangean submanifolds inside the
product, due to the following general fact.

Let $\left( M,\omega \right) $ and $\left( N,\omega _{1}\right) $ be
symplectic manifolds. The cartesian product $M\times N$ can be endowed with
the symplectic form $\omega \times \omega _{1}$. If $\phi \colon
M\rightarrow N$ is anti-symplectic that is, $\phi ^{\ast }\omega
_{1}=-\omega $ then $\mathrm{graph}\left( \phi \right) \subset M\times N$ is
a Lagrangean submanifold with respect to $\omega \times \omega _{1}$. 
Similarly, we could use a symplectic map (symplectomorphism) $\phi \colon
M\rightarrow M$ taking $\omega _{1}=-\omega $, which is also a symplectic
form.

With this in mind, to construct an assortment of Lagrangean submanifolds in $%
\mathcal{O}\left( H_{0}\right) $ we use an embedding $\mathcal{O}\left(
H_{0}\right) \hookrightarrow \mathbb{F}_{1}\times \mathbb{F}_{2}$ into a
product of flags. Taking symplectic forms $\omega _{1}$ and $\omega _{2}$ on 
$\mathbb{F}_{1}$ and $\mathbb{F}_{2}$ we obtain a symplectic form $\omega
_{1}\times \omega _{2}$ on $\mathbb{F}_{1}\times \mathbb{F}_{2}$ and
consequently on $\mathcal{O}\left( H_{0}\right) $ by restriction. If $\phi
\colon \mathbb{F}_{1}\rightarrow \mathbb{F}_{2}$ is anti-symplectic then $%
\mathrm{graph}\left( \phi \right) $ and $\mathrm{graph}\left( \phi \right)
\cap \mathcal{O}\left( H_{0}\right) $ are Lagrangean submanifolds of $%
\mathbb{F}_{1}\times \mathbb{F}_{2}$ and $\mathcal{O}\left( H_{0}\right) $,
respectively. %
The intended construction involves, first of all, a discussion about the
right action of the Weyl group.

\subsection{Right action of the Weyl group}

Let  $\mathfrak{g}$ be a noncompact semisimple Lie algebra (real or complex)
and let $G$ be the adjoint group of $\mathfrak{g}$ and $K\subset G$ the
maximal compact subgroup. The maximal flag of $\mathfrak{g}$ is given by $%
\mathbb{F}=G/P=K/M$, where $P=MAN$ is the minimal parabolic subgroup. The
adjoint orbit of a regular element $H\in \mathfrak{a}=\log A$ is given by $%
\mathcal{O}\left( H\right) =G/MA$. The flag $\mathbb{F}$ is contained in $%
\mathcal{O}\left( H\right) $ since $\mathbb{F}$ is a $K$-orbit of $H$.

The Weyl group $\mathcal{W}$ is isomorphic to $\mathrm{Norm}_{G}\left(
A\right) /MA=\mathrm{Norm}_{K}\left( A\right) /M$. We obtain right actions
of $\mathcal{W}$ on $\mathbb{F}=K/M$ (with $\mathcal{W}=\mathrm{Norm}%
_{K}\left( A\right) /M$) and on $\mathcal{O}\left( H\right) =G/MA$ (with $%
\mathcal{W}=\mathrm{Norm}_{G}\left( A\right) /MA$). Moreover, the fibrations 
$G/MA\rightarrow G/\mathrm{Norm}_{G}\left( A\right) $ and $\mathbb{F}%
=K/M\rightarrow K/\mathrm{Norm}_{K}\left( A\right) $ are principal bundles
with strutural group $\mathcal{W}$.

\begin{example}
Let $\mathfrak{g}=\mathfrak{sl}\left( n,\mathbb{R}\right) $ or $\mathfrak{g}=%
\mathfrak{sl}\left( n,\mathbb{C}\right) $. Hence, a regular element $H$ is a
diagonal matrix $H=\mathrm{diag}\{a_{1},\ldots ,a_{n}\}$ with $a_{1}>\cdots
>a_{n}$. $\mathcal{O}\left( H\right) =\{gHg^{-1}:g\in \mathrm{Sl}\left( n,%
\mathbb{R}\right) \}$ (or $\mathbb{C}$), that is, the orbit is the set of
diagonalizable matrices with the same eigenvalues as $H$. The Weyl group $%
\mathcal{W}$ is the permutation group of $n$ elements, whereas $\mathrm{Norm}%
_{K}\left( A\right) $ is the set of signed permutation matrices (matrices
such that each row or column has exactly one nonzero entry $\pm 1$). The
right action of a permutation $w\in \mathcal{W}$ is given by 
\begin{equation*}
R_{w}:gHg^{-1}\mapsto g\overline{w}H\left( g\overline{w}\right)
^{-1}=g\left( \overline{w}H\overline{w}^{-1}\right) g^{-1}
\end{equation*}%
where $\overline{w}\in \mathrm{Norm}_{K}\left( A\right) $ is the permutation
matrix that represents $w\in \mathcal{W}$. In this expression for $R_{w}$
the term $\overline{w}H\overline{w}^{-1}$ is the matrix whose diagonal
entries are the same as the ones of $H$ permuted by $w$ in the permutation
group $\mathcal{W}$. 

The right action $R_{w}$ of $w\in \mathcal{W}$ is in general completely
different from the left action of any of its representatives $\overline{w}%
\in \mathrm{Norm}_{K}\left( A\right) $. For example, in the case of $%
\mathfrak{sl}\left( 2,\mathbb{C}\right) $, the Weyl group is $\left\{
1,\left( 12\right) \right\} $ and the right action of $w=\left( 12\right) $
in the flag $S^{2}=\mathbb{C}P^{1}$ is the antipodal map. On the other hand, 
\begin{equation*}
\overline{w}=\left( 
\begin{array}{cc}
0 & -1 \\ 
1 & 0%
\end{array}%
\right) \in \mathrm{Norm}_{K}\left( A\right)
\end{equation*}%
is a representative of $\left( 12\right) $. The left action of $\overline{w}$
has 2 fixed points. 
\end{example}

The right action of $\mathcal{W}$ leaves invariant the induced vector fields:

\begin{proposition}
\label{propcampoinva}Given an element $A$ in the Lie algebra, denote by $%
\widetilde{A}$ the induced vector field on the homogeneous space ($G/MA$ or $%
K/M$). Then, $\left( R_{w}\right) _{\ast }\widetilde{A}=\widetilde{A}$ for
all $w\in \mathcal{W}$.
\end{proposition}

\begin{proof}
Indeed, $R_{w}$ commutes with the flow of $\widetilde{A}$, which is the left
action of $e^{tA}$.
\end{proof}

\subsection{The $K$-orbit and graphs}

In section \ref{secsplit}, we defined an embedding of the adjoint orbit into
the product $\mathbb{F}_{H_{0}}\times \mathbb{F}_{H_{0}^{\ast }}$, where $%
\mathbb{F}_{H_{0}^{\ast }}$ is the dual flag of $\mathbb{F}_{H_{0}}$.

Consider first of all the case when $\mathbb{F}_{H}=\mathbb{F}$ is the
maximal flag, which is self-dual. In this flag, the right action of $%
\mathcal{W}$ is well defined. Denote by $b_{0}$ the origin of $\mathbb{F}$
and set $b_{w}=R_{w}b_{0}$, $w\in \mathcal{W}$.

Let $w_{0}\in \mathcal{W}$ be the principal involution (element of largest
length as a product of simple reflections). The embedding of $\mathcal{O}%
\left( H_{0}\right) $ is given by the $G$-orbit of $\left(
b_{0},b_{w_{0}}\right) $ under the diagonal action $g\left( x,y\right)
=\left( gx,gy\right) $. This $G$-orbit is identified with the adjoint orbit $%
\mathcal{O}\left( H_{0}\right) =G/MA$ for any $H_{0}$ regular and real. Let $%
K$ be the maximal compact subgroup of $G$ (real compact form in the case of
complex $G$).

\begin{proposition}
\label{proporbitcompact}For $w\in \mathcal{W}$, the $K$-orbit of $\left(
b_{0},b_{w}\right) $ by the diagonal action coincides with the graph of $%
R_{w}$.
\end{proposition}

\begin{proof}
Take $x=k\cdot b_{0}\in \mathbb{F}$, $k\in K$. Then, $R_{w}\left( x\right)
=R_{w}\left( k\cdot b_{0}\right) =k\cdot R_{w}\left( b_{0}\right) $ since
the left and right actions commute. Thus, $\left( x,R_{w}\left( x\right)
\right) =\left( k\cdot b_{0},k\cdot R_{w}\left( b_{0}\right) \right) =k\cdot
\left( b_{0},b_{w}\right) $ belongs to the $K$-orbit of $\left(
b_{0},b_{w}\right) $. Conversely, an element of the orbit $k\cdot \left(
b_{0},b_{w}\right) =\left( x,R_{w}\left( x\right) \right) $, $x=k\cdot b_{0}$%
, belongs to the graph of $R_{w}$.
\end{proof}

\begin{remark}
In the case when $w=w_{0}$ is the principal involution, the $K$-orbit of
proposition \ref{proporbitcompact} corresponds to the zero section of $%
T^{\ast }\mathbb{F}$ when $\mathcal{O}\left( H_{0}\right) =G/MA$ is
identified with the cotangent bundle. This happens because the origin $G/MA$
gets mapped to $H_{0}\in \mathcal{O}\left( H_{0}\right) $ and the $K$-orbit
of $H_{0}$ is identified to the zero section. On the other hand, the origin
of the open orbit $G\cdot \left( b_{0},b_{w_{0}}\right) \in \mathbb{F}\times 
\mathbb{F}$ is $\left( b_{0},b_{w_{0}}\right) $, so that its $K$-orbit gets
identified to the $K$-orbit of $H_{0}$.
\end{remark}

\begin{remark}
It follows directly from proposition \ref{proporbitcompact} that the graphs
of right translations $R_{w}$, $w\in \mathcal{W}$, are contained in the
diagonal $G$-orbits and consequently are compact inside these orbits. This
does not happen with left translations, not even by elements of $\mathrm{Norm%
}_{K}\left( A\right) $, which represent elements of the Weyl group.
\end{remark}

\subsection{Example}

For $\mathfrak{sl}\left( 2,\mathbb{C}\right) $ the flag is $\mathbb{C}%
P^{1}=S^{2}$ and $\mathcal{W}=\{1,\left( 12\right) \}$. The right action of $%
R_{\left( 12\right) }$ on $S^{2}$ is the antipodal map. Another way to see
this right action is to identify $\mathbb{C}P^{1}$ with the set of Hermitian
matrices with eigenvalues $\pm 1$ (the adjoint orbit of the compact group $%
\mathrm{SU}\left( 2\right) $). This identification associates to a Hermitian
matrix the eigenspace associated to the eigenvalue $+1$. In this case, if $%
\xi =\langle \left( x,y\right) \rangle \in \mathbb{C}P^{1}$ then $R_{\left(
12\right) }\left( \xi \right) $ is the eigenspace of the Hermitian matrix
associated to the eigenvalue $-1$. That is, $R_{\left( 12\right) }\left( \xi
\right) $ is the Hermitian orthogonal $\xi ^{\bot }$ of $\xi $, which is
generated by $\left( -\overline{y},\overline{x}\right) $ if $\xi =\langle
\left( x,y\right) \rangle $.

It is convenient to write down the Hermitian matrix whose eigenspace
associated to $-1$ is $\xi =\langle \left( x,y\right) \rangle $ with $x%
\overline{x}+y\overline{y}=1$. It is given by 
\begin{equation}
\left( 
\begin{array}{cc}
x & -\overline{y} \\ 
y & \overline{x}%
\end{array}%
\right) \left( 
\begin{array}{cc}
1 & 0 \\ 
0 & -1%
\end{array}%
\right) \left( 
\begin{array}{cc}
\overline{x} & \overline{y} \\ 
-y & x%
\end{array}%
\right) =\left( 
\begin{array}{cc}
x\overline{x}-y\overline{y} & 2x\overline{y} \\ 
2\overline{x}y & -x\overline{x}+y\overline{y}%
\end{array}%
\right) .  \label{formatriz}
\end{equation}

Consider now the cartesian product $S^{2}\times S^{2}$ with the diagonal
action of $G=\mathrm{Sl}\left( 2,\mathbb{C}\right) $: $g\left( \xi ,\eta
\right) =\left( g\xi ,g\eta \right) $. There are 2 orbits:

\begin{enumerate}
\item the diagonal $\Delta =\{\left( \xi ,\xi \right) :\xi \in S^{2}\}$ and

\item and open and dense orbit $\{\left( \xi ,\eta \right) :\xi ,\eta \in
S^{2},\xi \neq \eta \}$. As a homogenous space of $G$ this open orbit is
given by $G/MA$ where $MA$ is the Cartan subgroup of diagonal matrices.
Thus, it can be identified with the adjoint orbit of 
\begin{equation*}
H_{0}=\left( 
\begin{array}{cc}
1 & 0 \\ 
0 & -1%
\end{array}%
\right) .
\end{equation*}%
This last identification is obtained explicitly associating $\left( \xi
,\eta \right) \in \mathbb{C}P^{1}\times \mathbb{C}P^{1}$ to the only $%
2\times 2$ matrix with eigenvalues $\pm 1$, whose eigenspaces associated to
eigenvalues $+1$ e $-1$ are $\xi $ and $\eta $, respectively. Therefore, if $%
\xi =\langle \left( x,y\right) \rangle $ and $\eta =\langle \left(
z,w\right) \rangle $ with $xw-yz=1$ then 
\begin{equation}
\left( \xi ,\eta \right) \mapsto \left( 
\begin{array}{cc}
x & z \\ 
y & w%
\end{array}%
\right) \left( 
\begin{array}{cc}
1 & 0 \\ 
0 & -1%
\end{array}%
\right) \left( 
\begin{array}{cc}
w & -z \\ 
-y & x%
\end{array}%
\right) =\left( 
\begin{array}{cc}
wx+yz & -2xz \\ 
2yw & -wx-yz%
\end{array}%
\right) .  \label{formatrizautoesp}
\end{equation}%
Here, the origin $o=MA$ of $G/MA$ is identified to $H_{0}$, in the adjoint
orbit and to $o=\left( \langle \left( 1,0\right) \rangle ,\langle \left(
0,1\right) \rangle \right) $, in the open orbit on $\mathbb{C}P^{1}\times 
\mathbb{C}P^{1}$. Accordingly, the open orbit of the diagonal action is $%
G\cdot o$.
\end{enumerate}

The right action $R_{\left( 12\right) }$ on $G/MA$ admits good descriptions
in terms of the identifications with the adjoint orbit $\mathrm{Ad}\left(
G\right) H_{0}$ and with the open orbit $G\cdot o$ in $S^{2}\times S^{2}$.
They go as follows:

\begin{enumerate}
\item If $A\in \mathrm{Ad}\left( G\right) H_{0}$ then $R_{\left( 12\right)
}\left( A\right) $ is the unique matrix $2\times 2$ with eigenvalues $\pm 1$
which has the same eigenspaces as those of $A$, but with the order of the
eigenvalues switched.

\item If $\left( \xi ,\eta \right) \in G\cdot o$ then $R_{\left( 12\right)
}\left( \xi ,\eta \right) =\left( \eta ,\xi \right) $, since in the first
case the order of the eigenspaces is switched.
\end{enumerate}

The diagonal $\Delta $ is also an orbit in the compact group $\mathrm{SU}%
\left( 2\right) $. Obviously $\Delta $ is the graph of the identity map of $%
S^{2}$. On the other hand, consider the $\mathrm{SU}\left( 2\right) $-orbit
through the origin $o=\left( \langle \left( 1,0\right) \rangle ,\langle
\left( 0,1\right) \rangle \right) $ of the open orbit in $\mathbb{C}%
P^{1}\times \mathbb{C}P^{1}$. This orbit is given by pairs $\left( \xi ,\eta
\right) $ orthogonal with respect to the canonical Hermitian form. From this
we deduce that $\mathrm{SU}\left( 2\right) \cdot o$ is the graph of the map $%
R_{\left( 12\right) }\colon S^{2}\rightarrow S^{2}$ (antipodal map in $S^{2}$%
). 
%

The Weyl group of the product $G\times G$ is the product $\mathcal{W}\times 
\mathcal{W}$ of the Weyl group $\mathcal{W}$ of $G$. In the case $G=\mathrm{%
Sl}\left( 2,\mathbb{C}\right) $, $\mathcal{W}\times \mathcal{W}$ has order $%
4 $, therefore its orbit through the origin $\left( \langle \left(
1,0\right) \rangle ,\langle \left( 1,0\right) \rangle \right) $ in the flag $%
\mathbb{C}P^{1}\times \mathbb{C}P^{1}$ has $4$ elements: two in the diagonal 
$\left( \langle \left( 1,0\right) \rangle ,\langle \left( 1,0\right) \rangle
\right) $ and $\left( \langle \left( 0,1\right) \rangle ,\langle \left(
0,1\right) \rangle \right) $ and two in the open orbit $\left( \langle
\left( 1,0\right) \rangle ,\langle \left( 0,1\right) \rangle \right) $ and $%
\left( \langle \left( 0,1\right) \rangle ,\langle \left( 1,0\right) \rangle
\right) $.

Take $w=\left( 1,\left( 12\right) \right) \in \mathcal{W}\times \mathcal{W}$%
. Then, $R_{w}=\mathrm{id}\times R_{\left( 12\right) }$ and a representative
of $w$ in the normalizer of the Cartan subgroup ($MA\times MA$) is $%
\overline{w}=\left( \mathrm{id},r\right) $ where $r$ is the rotation matrix 
\begin{equation*}
\left( 
\begin{array}{cc}
0 & -1 \\ 
1 & 0%
\end{array}%
\right) .
\end{equation*}%
This representative $\overline{w}$ acts on the left on the flag $\mathbb{C}%
P^{1}\times \mathbb{C}P^{1}$.

Consider the composition $\sigma =\overline{w}\circ R_{w}$, which is a
diffeomorphism of $S^{2}\times S^{2}$ that does not leave invariant the open
orbit although $\sigma \left( o\right) =o$. Since $\sigma $ acts only on the
second coordinate of the image of the flag of $G$, $\mathrm{SU}\left(
2\right) \cdot o$ (the orbit of the diagonal action) is given by the graph
of $r$: 
\begin{equation*}
\mathrm{graph}\left( r\right) =\{\left( \xi ,r\xi \right) :\xi \in \mathbb{C}%
P^{1}\}.
\end{equation*}%
In fact, if $\left( \xi ,R_{\left( 12\right) }\xi \right) \in \mathrm{SU}%
\left( 2\right) \cdot o=\mathrm{graph}\left( R_{\left( 12\right) }\right) $
then 
\begin{equation*}
\overline{w}\circ R_{w}\left( \xi ,R_{\left( 12\right) }\xi \right) =%
\overline{w}\left( \xi ,R_{\left( 12\right) }^{2}\xi \right) =\overline{w}%
\left( \xi ,\xi \right)
\end{equation*}%
since $R_{\left( 12\right) }^{2}=\mathrm{id}$. Therefore, $\overline{w}\circ
R_{w}\left( \xi ,R_{\left( 12\right) }\xi \right) =\left( \xi ,r\xi \right) $%
. Reciprocally $\sigma ^{-1}\left( \xi ,\eta \right) \in \mathrm{graph}%
\left( R_{\left( 12\right) }\right) $ if $\left( \xi ,\eta \right) \in 
\mathrm{graph}\left( r\right) $.

The image $\sigma \left( \mathrm{SU}\left( 2\right) \cdot o\right) =\mathrm{%
graph}\left( r\right) $ is not contained in the open $G$-orbit of the
diagonal action. The reason is that $r$ has two fixed points, that are the
subspaces generated by $\xi ^{+}=\left( 1,i\right) $ and $\xi ^{-}=\left(
1,-i\right) $. Hence $\left( \xi ^{\pm },\xi ^{\pm }\right) \in \mathrm{graph%
}\left( r\right) \cap \Delta $.

Since $\xi ^{\pm }$ are the only fixed points of $r$, 
\begin{equation*}
\mathrm{graph}\left( r\right) \cap G\cdot o=\{\left( \xi ,r\xi \right) :\xi
\neq \xi ^{\pm }\}.
\end{equation*}%
Take $\xi \neq \xi ^{\pm }$ generated by $\left( x,y\right) $. Then, $r\xi
\neq \xi $ and is generated by $\left( -y,x\right) $. The condition for $\xi
\neq \xi ^{\pm }$ is $x^{2}+y^{2}\neq 0$ which can be normalized to $%
x^{2}+y^{2}=1$. Using expression (\ref{formatrizautoesp}) the pair $\left(
\xi ,r\xi \right) $, $\xi \neq \xi ^{\pm }$, is identified to the matrix 
\begin{equation*}
\left( 
\begin{array}{cc}
x & -y \\ 
y & x%
\end{array}%
\right) \left( 
\begin{array}{cc}
1 & 0 \\ 
0 & -1%
\end{array}%
\right) \left( 
\begin{array}{cc}
x & y \\ 
-y & x%
\end{array}%
\right) =\left( 
\begin{array}{cc}
x^{2}-y^{2} & 2xy \\ 
2xy & -x^{2}+y^{2}%
\end{array}%
\right) .
\end{equation*}

Another example of graph is the map $m\circ R_{w}$ where 
\begin{equation*}
m=\left( 
\begin{array}{cc}
e^{i\pi } & 0 \\ 
0 & e^{-i\pi }%
\end{array}%
\right) =\left( 
\begin{array}{cc}
1 & 0 \\ 
0 & -1%
\end{array}%
\right) ,
\end{equation*}%
that is, 
\begin{equation*}
m\circ R_{w}\left[ \left( x,y\right) \right] =\left[ \left( \overline{y},%
\overline{x}\right) \right] .
\end{equation*}%
The straight lines $\left[ \left( 1,1\right) \right] $ e $\left[ \left(
1,-1\right) \right] $ are fixed points of $m\circ R_{w}$. Therefore, the
graph intercepts the diagonal, and consequently its intersection with $%
G\cdot o$ is not compact. By formula (\ref{formatrizautoesp}), if $%
|x|^{2}-|y|^{2}=1$ (so that the determinant is $1$) then matrices in $%
\mathcal{O}\left( H_{0}\right) $ have the form 
\begin{equation*}
\left( 
\begin{array}{cc}
x & \overline{y} \\ 
y & \overline{x}%
\end{array}%
\right) \left( 
\begin{array}{cc}
1 & 0 \\ 
0 & -1%
\end{array}%
\right) \left( 
\begin{array}{cc}
\overline{x} & -\overline{y} \\ 
-y & x%
\end{array}%
\right) =\left( 
\begin{array}{cc}
|x|^{2}+|y|^{2} & -2x\overline{y} \\ 
2y\overline{x} & -|x|^{2}-|y|^{2}%
\end{array}%
\right) .
\end{equation*}%
This is the intersection of the orbit with the subspace $\mathfrak{h}_{%
\mathbb{R}}$ plus the skew-Hermitian matrices with $0$'s in the diagonal. 

%

\subsection{Hermitian structures and symplectic forms}

Suppose here that $\mathfrak{g}$ is a complex algebra and take a Weyl basis $%
X_{\alpha }\in \mathfrak{g}_{\alpha }$. The real compact form $\mathfrak{u}$
is generated by $A_{\alpha }=X_{\alpha }-X_{-\alpha }$ and $Z_{\alpha
}=iS_{\alpha }=i\left( X_{\alpha }+X_{-\alpha }\right) $ with $\alpha >0$.
If $\mathfrak{u}_{\alpha }=\mathrm{span}\{A_{\alpha },Z_{\alpha }\}$ then
the tangent space at the origin $b_{H_{0}}$ of $\mathbb{F}_{H_{0}}$ is
isomorphic to 
\begin{equation*}
T_{H_{0}}=\sum_{\alpha \left( H_{0}\right) >0}\mathfrak{u}_{\alpha }
\end{equation*}%
via the isomorphism 
\begin{equation*}
Y\in T_{H_{0}}\mapsto \widetilde{Y}\left( b_{H_{0}}\right) =\frac{d}{dt}%
\left( e^{tY}\cdot b_{H_{0}}\right) _{\left\vert t=0\right. }\in
T_{b_{H_{0}}}\mathbb{F}_{H_{0}}.
\end{equation*}%
The canonical complex structure $J$ on $\mathbb{F}_{H_{0}}$ is invariant by
the compact group $K=\exp \mathfrak{u}$ and at the origin of the subspaces $%
\mathfrak{u}_{\alpha }$, $\alpha >0$, it is given by 
\begin{equation*}
JA_{\alpha }=Z_{\alpha }\qquad JZ_{\alpha }=-A_{\alpha }.
\end{equation*}%
%
%
%
%
%
%
%
%
%
%
%
%
%

\begin{proposition}
\label{propestrcomplxoutros}Let $\widetilde{w}$ be a representative of $w\in 
\mathcal{W}$. Then the tangent space to $\widetilde{w}H_{0}$ is identified
with 
\begin{equation*}
T_{\widetilde{w}H_{0}}=\sum_{\alpha \left( \widetilde{w}H_{0}\right) >0}%
\mathfrak{u}_{\alpha }
\end{equation*}%
and the canonical complex structure $J_{w}$ on $T_{\widetilde{w}H_{0}}$ is
given by 
\begin{equation*}
JA_{\alpha }=Z_{\alpha }\qquad JZ_{\alpha }=-A_{\alpha }
\end{equation*}%
with $A_{\alpha }$ and $Z_{\alpha }$, with the caveat that we take roots $%
\alpha $ such that $\alpha \left( \widetilde{w}H_{0}\right) >0$ (which are
not in general positive roots).
\end{proposition}

\begin{proof}
Since $J$ in invariant, $J_{w}=d\widetilde{w}_{H_{0}}\circ J_{0}\circ \left(
d\widetilde{w}_{H_{0}}\right) ^{-1}$, where for now $J_{0}$ denotes the
structure on $T_{H_{0}}$. Take $A_{\alpha }$ with $\alpha \left( \widetilde{w%
}H_{0}\right) >0$. Then, 
\begin{eqnarray*}
\left( d\widetilde{w}_{H_{0}}\right) ^{-1}\left( \widetilde{A}_{\alpha
}\left( \widetilde{w}H_{0}\right) \right) &=&d\left( \widetilde{w}%
^{-1}\right) _{H_{0}}\left( \widetilde{A}_{\alpha }\left( \widetilde{w}%
H_{0}\right) \right) \\
&=&\left( \mathrm{Ad}\left( \widetilde{w}^{-1}\right) A_{\alpha }\right)
^{\sim }\left( \widetilde{w}H_{0}\right) .
\end{eqnarray*}%
But,%
\begin{equation*}
\ \mathrm{Ad}\left( \widetilde{w}^{-1}\right) A_{\alpha }=\mathrm{Ad}\left( 
\widetilde{w}^{-1}\right) X_{\alpha }-\mathrm{Ad}\left( \widetilde{w}%
^{-1}\right) X_{-\alpha }=k\left( \widetilde{w}^{-1},\alpha \right)
X_{w^{-1}\alpha }-k\left( \widetilde{w}^{-1},-\alpha \right)
X_{-w^{-1}\alpha }.
\end{equation*}%
Applying $J_{0}$ (or rather its complexification) to this equality we obtain 
\begin{equation*}
J_{0}\circ \left( d\widetilde{w}_{H_{0}}\right) ^{-1}\widetilde{A}_{\alpha
}\left( \widetilde{w}H_{0}\right) =\left\{ 
\begin{array}{lll}
ik\left( \widetilde{w}^{-1},\alpha \right) X_{w^{-1}\alpha }+ik\left( 
\widetilde{w}^{-1},-\alpha \right) X_{-w^{-1}\alpha } & \mathrm{if} & 
w^{-1}\alpha >0 \\ 
-ik\left( \widetilde{w}^{-1},\alpha \right) X_{w^{-1}\alpha }-ik\left( 
\widetilde{w}^{-1},-\alpha \right) X_{-w^{-1}\alpha } & \mathrm{if} & 
w^{-1}\alpha <0.%
\end{array}%
\right.
\end{equation*}
Finally, applying $\mathrm{Ad}\left( \widetilde{w}\right) $ to the last term
we get 
\begin{equation*}
J_{w}\left( \widetilde{A}_{\alpha }\left( \widetilde{w}H_{0}\right) \right)
=\left\{ 
\begin{array}{lll}
iX_{\alpha }+iX_{-\alpha } & \mathrm{if} & w^{-1}\alpha >0 \\ 
-iX_{\alpha }-iX_{-\alpha } & \mathrm{if} & w^{-1}\alpha <0,%
\end{array}%
\right.
\end{equation*}
since $k\left( \widetilde{w},w^{-1}\alpha \right) k\left( \widetilde{w}%
^{-1},\alpha \right) =k\left( \widetilde{w},-w^{-1}\alpha \right) k\left( 
\widetilde{w}^{-1},-\alpha \right) =1$. A similar calculation gives 
\begin{equation*}
J_{w}\left( \widetilde{Z}_{\alpha }\left( \widetilde{w}H_{0}\right) \right)
=\left\{ 
\begin{array}{lll}
-\left( X_{\alpha }-X_{-\alpha }\right) & \mathrm{if} & w^{-1}\alpha >0 \\ 
X_{\alpha }-X_{-\alpha } & \mathrm{if} & w^{-1}\alpha <0.%
\end{array}%
\right.
\end{equation*}

However, $w^{-1}\alpha \left( H_{0}\right) =\alpha \left( wH_{0}\right) $,
which is $>0$, by hypothesis. Since $H_{0}$ belongs to the closure of the
Weyl chamber, we conclude that $w^{-1}\alpha \left( H_{0}\right) >0$.
\end{proof}

Every $K$-invariant Riemmannian metric on $\mathbb{F}_{H_{0}}$ is almost
Hermitian with respect to $J$ (see \cite{invher}). In general, the
corresponding K\"{a}hler form $\Omega $ is not closed and consequently not
symplectic. However, the K\"{a}hler form is symplectic for the case of the
Borel metric $\left( \cdot ,\cdot \right) ^{B}$, which is the $K$-invariant
metric defined at the origin by $\left( \mathfrak{u}_{\alpha },\mathfrak{u}%
_{\alpha }\right) ^{B}=0$ if $\alpha \neq \beta $ and satisfying 
\begin{equation*}
\left( \widetilde{A}_{\alpha }\left( H_{0}\right) ,\widetilde{A}_{\alpha
}\left( H_{0}\right) \right) _{H_{0}}^{B}=\left( \widetilde{Z}_{\alpha
}\left( H_{0}\right) ,\widetilde{Z}_{\alpha }\left( H_{0}\right) \right)
_{H_{0}}^{B}=\alpha \left( H_{0}\right)
\end{equation*}
\begin{equation*}
\left( \widetilde{A}_{\alpha }\left( H_{0}\right) ,\widetilde{Z}_{\alpha
}\left( H_{0}\right) \right) _{H_{0}}^{B}=0
\end{equation*}%
if $\alpha \left( H_{0}\right) >0$.

This description of the Borel metric also holds at other points of $\mathbb{F%
}_{H_{0}}=\mathrm{Ad}\left( U\right) \cdot H_{0}$. For example, the tangent
space at $\mathrm{Ad}\left( \widetilde{w}\right) \cdot H_{0}$ is $%
\sum_{\alpha \left( w\cdot H_{0}\right) >0}\mathfrak{u}_{\alpha }$ and the
metric at $\mathfrak{u}_{\alpha }$ is given by the same expression provided $%
\alpha \left( w\cdot H_{0}\right) >0$.

\begin{proposition}
The map $R_{w_{0}}\colon \mathbb{F}_{H_{0}}\rightarrow \mathbb{F}%
_{H_{0}^{\ast }}$ is an isometry of Borel metrics. 
\end{proposition}

\begin{proof}
Since $R_{w_{0}}$ is equivariant by the left actions on $\mathbb{F}_{H_{0}}$
and $\mathbb{F}_{H_{0}^{\ast }}$ and the metrics are $K$-invariant, it
suffices to verify the isometry at the origin. Equivariance also also
implies that $\left( R_{w_{0}}\right) _{\ast }\widetilde{A}=\widetilde{A}$.
Thus, 
\begin{equation*}
\left( \left( dR_{w_{0}}\right) _{x}\widetilde{A}\left( x\right) ,\left(
dR_{w_{0}}\right) _{x}\widetilde{B}\left( x\right) \right) _{R_{w_{0}}\left(
x\right) }^{B}=\left( \widetilde{A}\left( R_{w_{0}}\left( x\right) \right) ,%
\widetilde{B}\left( R_{w_{0}}\left( x\right) \right) \right)
_{R_{w_{0}}\left( x\right) }^{B}
\end{equation*}%
for $x\in \mathbb{F}_{H_{0}}$. At $x=H_{0}\in \mathbb{F}_{H_{0}}$ we have $%
R_{w_{0}}\left( H_{0}\right) =w_{0}H^{\ast }=-H_{0}$. Now, if $\alpha \left(
H_{0}\right) >0$, then 
\begin{equation*}
\left( \widetilde{A}_{\alpha }\left( H_{0}\right) ,\widetilde{A}_{\alpha
}\left( H_{0}\right) \right) _{H_{0}}^{B}=\alpha \left( H_{0}\right)
\end{equation*}%
and the second term of the previous equality for $A=B=A_{\alpha }$ is 
\begin{equation*}
\left( \widetilde{A}_{\alpha }\left( -H_{0}\right) ,\widetilde{A}_{\alpha
}\left( -H_{0}\right) \right) _{H_{0}}^{B}=-\alpha \left( -H_{0}\right)
=\alpha \left( H_{0}\right) .
\end{equation*}%
The same holds true for $Z_{\alpha }$ corresponding to any root $\alpha $
with $\alpha \left( H_{0}\right) >0$, so 
\begin{equation*}
\left( \left( dR_{w_{0}}\right) _{H_{0}}\widetilde{A}\left( H_{0}\right)
,\left( dR_{w_{0}}\right) _{H_{0}}\widetilde{B}\left( H_{0}\right) \right)
_{-H_{0}}^{B}=\left( \widetilde{A}\left( H_{0}\right) ,\widetilde{B}\left(
H_{0}\right) \right) _{H_{0}}^{B}
\end{equation*}%
for arbitrary $A$ and $B$. This shows that $R_{w_{0}}$ is an isometry.
\end{proof}

Having obtained the isometry $R_{w_{0}}$, its holomorphicity provides us
with the symplectic isomorphism.

\begin{proposition}
\label{ah} The map $R_{w_{0}}\colon \mathbb{F}_{H_{0}}\rightarrow \mathbb{F}%
_{H_{0}^{\ast }}$ is \textbf{anti}-holomorphic with respect to the canonical
complex structures on $\mathbb{F}_{H_{0}}$ and $\mathbb{F}_{H_{0}^{\ast }}$.
\end{proposition}

\begin{proof}
Let $\widetilde{w}_{0}$ be a representative of $w_{0}$ such that $\mathrm{Ad}%
\left( \widetilde{w}_{0}\right) H_{0}^{\ast }=-H_{0}$ and denote by $J_{0}$
and $J_{w_{0}}$ the complex structures no the tangent spaces $T_{H_{0}}%
\mathbb{F}_{H_{0}}$ and $T_{-H_{0}}\mathbb{F}_{H_{0}^{\ast }}$,
respectively. Take a root $\alpha $ com $\alpha \left( H_{0}\right) >0$,
that is, $\left( -\alpha \right) \left( -H_{0}\right) >0$. By proposition %
\ref{propestrcomplxoutros} we have 
\begin{equation*}
J_{0}\left( \widetilde{A}_{\alpha }\left( H_{0}\right) \right) =\widetilde{Z}%
_{\alpha }\left( H_{0}\right) \qquad J_{0}\left( \widetilde{Z}_{\alpha
}\left( H_{0}\right) \right) =-\widetilde{A}_{\alpha }\left( H_{0}\right)
\end{equation*}%
since $\alpha \left( H_{0}\right) >0$, and 
\begin{equation*}
J_{w_{0}}\left( \widetilde{A}_{\alpha }\left( -H_{0}\right) \right)
=-J_{w_{0}}\left( \widetilde{A}_{-\alpha }\left( -H_{0}\right) \right) =-%
\widetilde{Z}_{\alpha }\left( -H_{0}\right) 
\end{equation*}
\begin{equation*}
 J_{w_{0}}\left( 
\widetilde{Z}_{\alpha }\left( -H_{0}\right) \right) =-\widetilde{A}_{-\alpha
}\left( -H_{0}\right) =\widetilde{A}_{\alpha }\left( -H_{0}\right)
\end{equation*}%
since $\left( -\alpha \right) \left( -H_{0}\right) >0$.

On the other hand, $\left( R_{w_{0}}\right) _{\ast }\widetilde{A}_{\alpha }=%
\widetilde{A}_{\alpha }$ and $\left( R_{w_{0}}\right) _{\ast }\widetilde{Z}%
_{\alpha }=\widetilde{Z}_{\alpha }$. Therefore,%
\begin{eqnarray*}
J_{w_{0}}\left( \left( dR_{w_{0}}\right) _{H_{0}}\widetilde{A}_{\alpha
}\left( H_{0}\right) \right) &=&J_{w_{0}}\left( \widetilde{A}_{\alpha
}\left( -H_{0}\right) \right) =-\widetilde{Z}_{\alpha }\left( -H_{0}\right)
\\
J_{w_{0}}\left( \left( dR_{w_{0}}\right) _{H_{0}}\widetilde{Z}_{\alpha
}\left( H_{0}\right) \right) &=&J_{w_{0}}\left( \widetilde{Z}_{\alpha
}\left( -H_{0}\right) \right) =\widetilde{A}_{\alpha }\left( -H_{0}\right)
\end{eqnarray*}%
whereas 
\begin{eqnarray*}
\left( dR_{w_{0}}\right) _{H_{0}}J_{0}\left( \widetilde{A}_{\alpha }\left(
H_{0}\right) \right) &=&\widetilde{Z}_{\alpha }\left( -H_{0}\right) \\
\left( dR_{w_{0}}\right) _{H_{0}}J_{0}\left( \widetilde{Z}_{\alpha }\left(
H_{0}\right) \right) &=&-\widetilde{A}_{\alpha }\left( -H_{0}\right)
\end{eqnarray*}%
which shows that $R_{w_{0}}$ is anti-holomorphic at the origin, and
consequently, on the whole flag by the invariance of the complex structures.
\end{proof}

\begin{corollary}
If $k\in K$ then the composites $R_{w_{0}}\circ k$ e $k\circ R_{w_{0}}$ are
anti-holomorphic.
\end{corollary}

\begin{corollary}
\label{corantisimpl}Let $\Omega _{H_{0}}$ and $\Omega _{H_{0}^{\ast }}$ be
the K\"{a}hler forms of the Hermitian structures on $\mathbb{F}_{H_{0}}$ and 
$\mathbb{F}_{H_{0}^{\ast }}$ given by the Borel metric and the canonical
complex structures. Then, $R_{w_{0}}$ is anti-symplectic, that is, $%
R_{w_{0}}^{\ast }\Omega _{H_{0}^{\ast }}=-\Omega _{H_{0}}$.
\end{corollary}

\subsection{Hermitian structures on products}

The product $\mathbb{F}_{H_{0}}\times \mathbb{F}_{H_{0}^{\ast }}$ is a flag
of the product $G\times G$ associated to $\left( H_{0},H_{0}^{\ast }\right) $%
, that is, $\mathbb{F}_{H_{0}}\times \mathbb{F}_{H_{0}^{\ast }}=\mathbb{F}%
_{\left( H_{0},H_{0}^{\ast }\right) }$. This flag has Borel metric and
invariant complex structures.

The adjoint orbit $\mathcal{O}\left( H_{0}\right) $ is identified to the
orbit $G\cdot \left( H_{0},-H_{0}\right) $ by the diagonal representation
(recall that $-H_{0}\in \mathbb{F}_{H_{0}^{\ast }}$ since $\mathrm{Ad}\left( 
\widetilde{w}_{0}\right) H_{0}^{\ast }=-H_{0}$ if $\widetilde{w}_{0}$ is a
representative of $w_{0}$). The adjoint orbit $\mathcal{O}\left(
H_{0}\right) $ has a complex structure inherited from the inclusion into $%
\mathfrak{g}$. 
On the other hand, the graphs considered above are Lagrangean with respect
to a symplectic form defined from the complex structures of the flags.
Hence, to continue our analysis we must compare these different complex
structures.

We take $\mathfrak{h}\times \mathfrak{h}$ as a Cartan subalgebra in $%
\mathfrak{g}\times \mathfrak{g}$. The roots of $\mathfrak{h}\times \mathfrak{%
h}$ are those of $\mathfrak{h}$ in each component and the root spaces are of
the form $\mathfrak{g}_{\alpha }\times \{0\}$ or $\{0\}\times \mathfrak{g}%
_{\alpha }$.

The tangent space $T_{\left( H_{0},-H_{0}\right) }\mathbb{F}_{\left(
H_{0},H_{0}^{\ast }\right) }$ is generated by $\left( A_{\alpha },0\right) $%
, $\left( Z_{\alpha },0\right) $, $\left( 0,A_{\alpha }\right) $ and $\left(
0,Z_{\alpha }\right) $. To obtain these generators, we can take the positive
roots $\alpha >0$. Here, if $\alpha $ is a positive root, then $\alpha
\left( H_{0}\right) >0$ but $\alpha \left( -H_{0}\right) <0$, determining a
difference between the complex structures of the first and second components.

In fact, if $\alpha >0$ and $\left( \mathfrak{u}\times \mathfrak{u}\right)
_{\alpha }$ denotes the space generated by the $4$ vectors above, then the
canonical complex structure on $\left( \mathfrak{u}\times \mathfrak{u}%
\right) _{\alpha }\subset T_{\left( H_{0},-H_{0}\right) }\mathbb{F}_{\left(
H_{0},H_{0}^{\ast }\right) }$ is given by 
\begin{equation}
\begin{array}{ccc}
J\left( A_{\alpha },0\right) =\left( Z_{\alpha },0\right) &  & J\left(
Z_{\alpha },0\right) =-\left( A_{\alpha },0\right) \\ 
J\left( 0,A_{\alpha }\right) =-\left( 0,Z_{\alpha }\right) &  & J\left(
0,Z_{\alpha }\right) =\left( 0,A_{\alpha }\right) .%
\end{array}
\label{forjotaprod}
\end{equation}

\begin{remark}
These expressions show that the canonical complex structure on $\mathbb{F}%
_{\left( H_{0},H_{0}^{\ast }\right) }=\mathbb{F}_{H_{0}}\times \mathbb{F}%
_{H_{0}^{\ast }}$ is the product of the canonical complex structures on $%
\mathbb{F}_{H_{0}}$ and $\mathbb{F}_{H_{0}^{\ast }}$.
\end{remark}

Another basis of the tangent space $T_{\left( H_{0},-H_{0}\right) }\mathbb{F}%
_{\left( H_{0},H_{0}^{\ast }\right) }$is given by the vectors $\left( 
\widetilde{X}_{-\alpha }\left( H_{0}\right) ,0\right) $, $\left( \widetilde{%
iX}_{-\alpha }\left( H_{0}\right) ,0\right) $, $\left( 0,\widetilde{X}%
_{\alpha }\left( -H_{0}\right) \right) $ e $\left( 0,\widetilde{iX}_{\alpha
}\left( -H_{0}\right) \right) $ with $\alpha $ running over the \textbf{%
positive} roots (since $T_{H_{0}}\mathbb{F}_{H_{0}}$ is generated by the
first two vectors and $T_{-H_{0}}\mathbb{F}_{H_{0}^{\ast }}$ by the second
pair). Still taking \textbf{positive} roots, these vectors satisfy:

\begin{enumerate}
\item $\left( \widetilde{X}_{-\alpha }\left( H_{0}\right) ,0\right) =-\left( 
\widetilde{A}_{\alpha }\left( H_{0}\right) ,0\right) $ since $A_{\alpha
}=X_{\alpha }-X_{-\alpha }$ and $\widetilde{X}_{\alpha }\left( H_{0}\right)
=0$.

\item $\left( \widetilde{iX}_{-\alpha }\left( H_{0}\right) ,0\right) =\left( 
\widetilde{Z}_{\alpha }\left( H_{0}\right) ,0\right) $ since $Z_{\alpha
}=iX_{\alpha }+iX_{-\alpha }$ and $\widetilde{iX}_{\alpha }\left(
H_{0}\right) =0$.

\item $\left( 0,\widetilde{X}_{\alpha }\left( -H_{0}\right) \right) =\left(
0,\widetilde{A}_{\alpha }\left( -H_{0}\right) \right) $ since $A_{\alpha
}=X_{\alpha }-X_{-\alpha }$ and $\widetilde{X}_{-\alpha }\left(
-H_{0}\right) =0$.

\item $\left( 0,\widetilde{iX}_{\alpha }\left( -H_{0}\right) \right) =\left(
0,\widetilde{Z}_{\alpha }\left( -H_{0}\right) \right) $ since $Z_{\alpha
}=iX_{\alpha }+iX_{-\alpha }$ and $\widetilde{iX}_{-\alpha }\left(
-H_{0}\right) =0$.
\end{enumerate}

Therefore, 
\begin{equation}
\begin{array}{ccc}
J\left( \widetilde{X}_{-\alpha }\left( H_{0}\right) ,0\right) =-\left( 
\widetilde{iX}_{-\alpha }\left( H_{0}\right) ,0\right) &  & J\left( 
\widetilde{iX}_{-\alpha }\left( H_{0}\right) ,0\right) =\left( \widetilde{X}%
_{-\alpha }\left( H_{0}\right) ,0\right) \\ 
J\left( 0,\widetilde{X}_{\alpha }\left( -H_{0}\right) \right) =-\left( 0,%
\widetilde{iX}_{\alpha }\left( -H_{0}\right) \right) &  & J\left( 0,%
\widetilde{iX}_{\alpha }\left( -H_{0}\right) \right) =\left( 0,\widetilde{X}%
_{\alpha }\left( -H_{0}\right) \right) .%
\end{array}
\label{forjotacanon}
\end{equation}

Using these calculations we obtain the following statement.

\begin{proposition}
Let $J^{\mathrm{in}}$ and $J$ be the following complex structures on $%
\mathcal{O}\left( H_{0}\right) \approx G\cdot \left( H_{0},-H_{0}\right) $:

\begin{enumerate}
\item $J^{\mathrm{in}}$ is the complex structure on $\mathcal{O}\left(
H_{0}\right) \subset \mathfrak{g}$ inherited from $\mathfrak{g}$ 

\item $J$ is the complex structure on $G\cdot \left( H_{0},-H_{0}\right) $
obtained by restriction of the complex structure on $\mathbb{F}%
_{H_{0}}\times \mathbb{F}_{H_{0}^{\ast }}$, defined at the origin by (\ref%
{forjotacanon}).
\end{enumerate}

Then $J^{\mathrm{in}}=-J$.
\end{proposition}

\begin{proof}
It suffices to verify that equality holds at the origin, since both complex
structures are $G$-invariant. The tangent space to $\mathcal{O}\left(
H_{0}\right) $ at the origin is generated by $\widetilde{W}\left(
H_{0}\right) =[W,H_{0}]$ with $W$ in $\mathfrak{g}_{\pm \alpha }$ and $%
\alpha $ running over all positive roots. If $W\in \mathfrak{g}_{-\alpha }$, 
$\alpha >0$, then $\widetilde{W}\left( H_{0}\right) $ is \textquotedblleft
horizontal \textquotedblright\ in the identification with $G\cdot \left(
H_{0},-H_{0}\right) $ whereas $\widetilde{W}\left( H_{0}\right) $ is
\textquotedblleft vertical\textquotedblright\ if $W\in \mathfrak{g}_{\alpha
} $, $\alpha >0$. For the complex structure on $\mathfrak{g}$ we have $%
X_{\alpha }\mapsto iX_{\alpha }$ and $iX_{\alpha }\mapsto -X_{\alpha }$.
Thus the complex structure $J^{\mathrm{in}}$ is given in the product by 
\begin{equation*}
\begin{array}{cc}
J^{\mathrm{in}}\left( \widetilde{X}_{-\alpha }\left( H_{0}\right) ,0\right)
=\left( \widetilde{iX}_{-\alpha }\left( H_{0}\right) ,0\right) &   J^{%
\mathrm{in}}\left( \widetilde{iX}_{-\alpha }\left( H_{0}\right) ,0\right)
=-\left( \widetilde{X}_{-\alpha }\left( H_{0}\right) ,0\right) \\ 
J^{\mathrm{in}}\left( 0,\widetilde{X}_{\alpha }\left( -H_{0}\right) \right)
=\left( 0,\widetilde{iX}_{\alpha }\left( -H_{0}\right) \right) &   J^{%
\mathrm{in}}\left( 0,\widetilde{iX}_{\alpha }\left( -H_{0}\right) \right)
=-\left( 0,\widetilde{X}_{\alpha }\left( -H_{0}\right) \right) ,%
\end{array}%
\end{equation*}%
which is precisely the negative of (\ref{forjotacanon}).
\end{proof}

Let $\left( \cdot ,\cdot \right) ^{B}$ be the Borel metric on $\mathbb{F}%
_{H_{0}}\times \mathbb{F}_{H_{0}^{\ast }}=\mathbb{F}_{\left(
H_{0},H_{0}^{\ast }\right) }$. If follows immediately from the definition
that $\left( \cdot ,\cdot \right) ^{B}$ is the product of the Borel metrics
on $\mathbb{F}_{H_{0}}$ and $\mathbb{F}_{H_{0}^{\ast }}$.

This metric together with the canonical complex structure $J$, define a
Hermitian structure on $\mathbb{F}_{H_{0}}\times \mathbb{F}_{H_{0}^{\ast }}$%
, which is invariant by $K\times K$ (compact group) but is not invariant by $%
G\times G$, because the metric itself is only invariant by $K\times K$. This
Hermitian structures restricts to a Hermitian structure in the open orbit $%
G\cdot \left( H_{0},-H_{0}\right) \approx \mathcal{O}\left( H_{0}\right) $,
which is invariant by the action of $K$ (but not by that of $G$). Denote by $%
\Omega \left( \cdot ,\cdot \right) =\left( \cdot ,J\left( \cdot \right)
\right) ^{B}$ the corresponding K\"ahler form, which is a symplectic form.
Since $\left( \cdot ,\cdot \right) ^{B}$ is the product metric and $J$ the
product complex structure, it follows that $\Omega $ is the product of the K%
\"{a}hler forms in $\mathbb{F}_{H_{0}}$ and $\mathbb{F}_{H_{0}^{\ast }}$.

\section{Lagrangean graphs in products of flags \label{seclagraph}}

By corollary \ref{corantisimpl} the map $R_{w_{0}}\colon \mathbb{F}%
_{H_{0}}\rightarrow \mathbb{F}_{H_{0}^{\ast }}$ is anti-symplectic with
respect to the K\"{a}hler forms on $\mathbb{F}_{H_{0}}$ and $\mathbb{F}%
_{H_{0}^{\ast }}$ given by the Borel metric and canonical complex
structures. Therefore, $\mathrm{graph}\left( R_{w_{0}}\right) $ is a
Lagrangean submanifold of the product symplectic structure. 
We now obtain further examples of Lagrangean graphs by composites (either on
the left or on the right) of $R_{w_{0}}$ with symplectic maps.

\begin{example}
If $k_{1},k_{2}\in K$ then the induced maps $k_{1}\colon \mathbb{F}%
_{H_{0}}\rightarrow \mathbb{F}_{H_{0}}$ and $k_{2}\colon \mathbb{F}%
_{H_{0}^{\ast }}\rightarrow \mathbb{F}_{H_{0}^{\ast }}$ are symplectic.
Therefore, $k_{1}\circ R_{w_{0}}\circ k_{2}$ is anti-symplectic, hence its
graph is a Lagrangean submanifold of $\mathbb{F}_{H_{0}}\times \mathbb{F}%
_{H_{0}^{\ast }}=\mathbb{F}_{\left( H_{0},H_{0}^{\ast }\right) }$. Such
graph is not contained in $G\cdot \left( H_{0},-H_{0}\right) $, nevertheless
its intersection with the orbit is still a Lagrangean submanifold
(noncompact if the graph is not contained in the orbit).
\end{example}

The tangent space $T_{\left( x,\phi \left( x\right) \right) }\mathrm{graph}%
\left( \phi \right) $ is given by the vectors $\left( u,d\phi _{x}\left(
u\right) \right) $. For maps $k\circ R_{w_{0}}$, with $k\in K$, the tangent
spaces admit the following description in terms of the adjoint
representation.

\begin{proposition}
\label{propesptangraph} Let $k\in K$. The tangent space to $\mathrm{graph}%
\left( k\circ R_{w_{0}}\right) $ at $\left( x,y\right) =\left( x,k\circ
R_{w_{0}}\left( x\right) \right) $ is given by 
\begin{equation*}
\{\left( A,\mathrm{Ad}\left( k\right) A\right) ^{\sim }\left( x,k\circ
R_{w_{0}}\left( x\right) \right) :A\in \mathfrak{u}\}
\end{equation*}%
where $\left( A,\mathrm{Ad}\left( k\right) A\right) ^{\sim }$ is the vector
field on $\mathbb{F}_{H_{0}}\times \mathbb{F}_{H_{0}^{\ast }}=\mathbb{F}%
_{\left( H_{0},H_{0}^{\ast }\right) }$ induced by $\left( A,\mathrm{Ad}%
\left( k\right) A\right) \in \mathfrak{u}\times \mathfrak{u}$ ($\mathfrak{u}$
= Lie algebra of $K$).
\end{proposition}

\begin{proof}
If $A\in \mathfrak{u}$ then $\left( R_{w_{0}}\right) _{\ast }\widetilde{A}=%
\widetilde{A}$, thus $\left( dR_{w_{0}}\right) _{x}\left( \widetilde{A}%
\left( x\right) \right) =\widetilde{A}\left( R_{w}\left( x\right) \right) $.
Applying $dk_{R_{w}\left( x\right) }$ to this equality we get 
\begin{eqnarray*}
\left( dk\circ R_{w_{0}}\right) _{x}\left( \widetilde{A}\left( x\right)
\right) &=&dk_{R_{w}\left( x\right) }\left( \widetilde{A}\left( R_{w}\left(
x\right) \right) \right) \\
&=&\widetilde{\mathrm{Ad}\left( k\right) A}\left( k\circ R_{w_{0}}\left(
x\right) \right) .
\end{eqnarray*}%
It follows that the tangent space to the graph is $\left( \widetilde{A}%
\left( x\right) ,\widetilde{\mathrm{Ad}\left( k\right) A}\left( k\circ
R_{w_{0}}\left( x\right) \right) \right) $. But the action of $K\times U $
on $\mathbb{F}_{H_{0}}\times \mathbb{F}_{H_{0}^{\ast }}$ works
coordinatewise. Hence
\begin{equation*}
\left( \widetilde{A}\left( x\right) ,\widetilde{\mathrm{Ad}\left( k\right) A}%
\left( k\circ R_{w_{0}}\left( x\right) \right) \right) =\left( A,\mathrm{Ad}%
\left( k\right) A\right) ^{\sim }\left( x,k\circ R_{w_{0}}\left( x\right)
\right)
\end{equation*}%
which completes the proof, because the vectors $\widetilde{A}\left( x\right) 
$, $A\in \mathfrak{u}$, exhaust the tangent space at $x$.
\end{proof}

\subsection{Graphs in $\mathbb{F}_{H_{\protect\mu }}\times \mathbb{F}_{H_{%
\protect\mu }^{\ast }}$}

The isomorphism between the open orbit in $\mathbb{F}_{H_{\mu }}\times 
\mathbb{F}_{H_{\mu }^{\ast }}$ (diagonal action) and the orbit $G\cdot
\left( v_{0}\otimes \varepsilon _{0}\right) $ of $v_{0}\otimes \varepsilon
_{0}\in V\otimes V^{\ast }$ (representation of $G$) leads to a convenient
description of the intersection of graphs of anti-holomorphic functions $%
\mathbb{F}_{H_{\mu }}\rightarrow \mathbb{F}_{H_{\mu }^{\ast }}$ with the
open orbit.

We return to the anti-holomorphic functions considered earlier $m\circ
R_{w_{0}}:\mathbb{F}_{H_{\mu }}\rightarrow \mathbb{F}_{H_{\mu }^{\ast }}$
with $m\in T$, the maximal torus. The submanifold determined by $\mathrm{%
graph}\left( R_{w_{0}}\right) $ in $R_{w_{0}}$ on $\mathbb{F}_{H_{\mu
}}\times \mathbb{F}_{H_{\mu }^{\ast }}$ is the orbit of the compact group $K$
through $\left( v_{0},\varepsilon _{0}\right) $. This orbit stays inside $%
G\cdot \left( v_{0},\varepsilon _{0}\right) $ and is identified with the $K$%
-orbit of $v_{0}\otimes \varepsilon _{0}$ in $V\otimes V^{\ast }$ (by
equivariance). The isomorphism with the adjoint orbit $\mathrm{Ad}\left(
G\right) H_{\mu }$ associates this $K$-orbit inside $V\otimes V^{\ast }$
with the intersection \ $i\mathfrak{u}\cap $\textrm{Ad}$\left( G\right)
H_{\mu }$ (the Hermitian matrices in the case of $\mathfrak{sl}\left( n,%
\mathbb{C}\right) $ or else the zero section of $T^{\ast }\mathbb{F}_{H_{\mu
}}$). This set is formed by the elements $v\otimes \varepsilon \in G\cdot
\left( v_{0}\otimes \varepsilon _{0}\right) $ such that $\ker \varepsilon
=v^{\bot }$ (with respect to the $K$-invariant Hermitian form $\left( \cdot
,\cdot \right) ^{\mu }$ ), since $u\in K$ is an isometry of $\left( \cdot
,\cdot \right) ^{\mu }$ and $\ker \varepsilon _{0}=v_{0}^{\bot }$. The
converse is true as well: if $v\otimes \varepsilon \in G\cdot \left(
v_{0}\otimes \varepsilon _{0}\right) $ and $\ker \varepsilon =v^{\bot }$
then $v\otimes \varepsilon \in \mathrm{graph}\left( R_{w_{0}}\right) $. In
fact, if $\ker \varepsilon =v^{\bot }$ and $X\in \mathfrak{u}$ then $\rho
_{\mu }\left( X\right) $ is anti-Hermitian, thus $\left( \rho _{\mu }\left(
X\right) v,v\right) ^{\mu }$ is purely imaginary and since $\ker \varepsilon
=v^{\bot }$, then $\varepsilon \left( \rho _{\mu }\left( X\right) v\right) $
is purely imaginary as well. Therefore, $\langle M\left( v\otimes
\varepsilon \right) ,X\rangle =\varepsilon \left( \rho _{\mu }\left(
X\right) v\right) $ is imaginary for arbitrary $X\in \mathfrak{u}$, which
implies that $M\left( v\otimes \varepsilon \right) \in i\mathfrak{u}$.

Summing up, we obtain the following description of $\mathrm{graph}\left(
R_{w_{0}}\right) $ regarded as a subset of $G\cdot \left( v_{0}\otimes
\varepsilon _{0}\right) $. Consider $\Phi ^{-1}\left( \mathrm{graph}\left(
R_{w_{0}}\right) \right) \subset G\cdot \left( v_{0}\otimes \varepsilon
_{0}\right) $, which, abusing notation, we also denoted by $\mathrm{graph}%
\left( R_{w_{0}}\right) $:

\begin{proposition}
$\mathrm{graph}\left( R_{w_{0}}\right) =\{v\otimes \varepsilon \in G\cdot
\left( v_{0}\otimes \varepsilon _{0}\right) :\ker \varepsilon =v^{\bot }\}.$ 
\end{proposition}

Consider now the graph of $m\circ R_{w_{0}}:\mathbb{F}_{H_{\mu }}\rightarrow 
\mathbb{F}_{H_{\mu }^{\ast }}$ with $m\in T$. In general $\mathrm{graph}%
\left( m\circ R_{w_{0}}\right) \subset \mathbb{F}_{H_{\mu }}\times \mathbb{F}%
_{H_{\mu }^{\ast }}$ is not contained in the open orbit and, consequently,
intercepts this orbit in a noncompact subset. In any case, take the subgroup 
\begin{equation*}
U^{m}=\{\left( u,mum^{-1}\right) \in U\times U:u\in U\}.
\end{equation*}%
Then, $\mathrm{graph}\left( m\circ R_{w_{0}}\right) $ is the orbit of $K^{m} 
$ through $\left( v_{0},\varepsilon _{0}\right) $. This happens because, if $%
x=u\cdot v_{0}\in \mathbb{F}_{H_{\mu }}$ then $R_{w_{0}}\left( x\right)
=u\cdot \varepsilon _{0}$ therefore 
\begin{equation*}
\left( x,m\circ R_{w_{0}}\left( x\right) \right) =\left( x,m\cdot
u\varepsilon _{0}\right) .
\end{equation*}%
This means that $\mathrm{graph}\left( m\circ R_{w_{0}}\right) $ is formed by
elements of the form $\left( x,my\right) $ with $\left( x,y\right) \in 
\mathrm{graph}\left( R_{w_{0}}\right) $, that is, 
\begin{equation*}
\mathrm{graph}\left( m\circ R_{w_{0}}\right) =m_{2}\left( \mathrm{graph}%
\left( R_{w_{0}}\right) \right)
\end{equation*}%
where $m_{2}\left( x,y\right) =\left( y,mx\right) $. Passing to the
realization inside $V\otimes V^{\ast }$ we obtain a geometric realization of 
$\Phi ^{-1}\left( \mathrm{graph}\left( m\circ R_{w_{0}}\right) \right) $,
also denoted by $\mathrm{graph}\left( m\circ R_{w_{0}}\right) $:

\begin{proposition}
$\mathrm{graph}\left( m\circ R_{w_{0}}\right) =\{v\otimes \rho _{\mu }^{\ast
}\left( m\right) \varepsilon \in G\cdot \left( v_{0}\otimes \varepsilon
_{0}\right) :\ker \varepsilon =v^{\bot }\}.$
\end{proposition}

In conclusion, we have described the following families of Lagrangean
submanifolds of the adjoint orbit $\mathcal{O}\left( H_{\Theta }\right) =%
\mathrm{Ad}\left( G\right) \cdot H_{\Theta }\approx G/Z_{\Theta }$:

\begin{theorem}
For $k_1,k_2 \in K$ and for $m \in T$:

\begin{itemize}
\item $\mathrm{graph}\left(k_1\circ R_{w_{0}}\circ k_{2}\right)$ corresponds
to a Lagrangean submanifold of $\mathcal{O}\left( H_{\Theta }\right)$, and

\item $\mathrm{graph}\left(m\circ R_{w_{0}}\right)$ corresponds to a
Lagrangean submanifold of $\mathcal{O}\left( H_{\Theta }\right).$
\end{itemize}
\end{theorem}

\end{document}